\documentclass[a4paper, 12pt, reqno]{amsart}
\usepackage{amsmath,amsfonts,amssymb,amsthm,enumerate}
\usepackage{graphicx}
 \usepackage{xy} \xyoption{all}

\newtheorem{thm}{Theorem}[section]
\newtheorem{cor}[thm]{Corollary}
\newtheorem{prop}[thm]{Proposition}
\newtheorem{lem}[thm]{Lemma}

\theoremstyle{definition}
\newtheorem{defn}[thm]{Definition}
\newtheorem{exas}[thm]{Examples}
\newtheorem{example}[thm]{Example}
\newtheorem{rem}[thm]{Remark}

\let\phi\varphi

\begin{document}
\title{Leavitt path algebras having \\ Unbounded Generating Number}
\maketitle
%\vspace{-.25in}
\begin{center}
G.~Abrams\footnote{Department of Mathematics, University of Colorado, Colorado Springs,
Coloralo, USA. E-mail address: \texttt{abrams@math.uccs.edu}},
T.\,G.~Nam\footnote{Institute of Mathematics, VAST, 18 Hoang Quoc Viet, Cau Giay, Hanoi, Vietnam. E-mail address: \texttt{tgnam@math.ac.vn}} and N.\,T.~Phuc\footnote{Department of IT and Mathematics Teacher Training, Dong Thap University, Vietnam. E-mail address: \texttt{ntphuc@dthu.edu.vn}    

\ \ \ {\bf Acknowledgements}:   
The first  author is partially supported by a Simons Foundation Collaboration Grants for
Mathematicians Award \#208941;   the second author is supported by Vietnam National Foundation for Science and Technology Development (NAFOSTED);    the third author is supported by project the CS2015.01.28.   The authors  express their deep gratitude to the editor for his 
careful and  professional shepherding  of the original version of this manuscript.   We express our thanks as well to Professor P. N. \'{A}nh of the Alfr\'{e}d R\'{e}nyi
Mathematical Institute, Hungarian Academy of Sciences,  for his valuable suggestions which led to the final shape of the paper.
 } %Research supported by Vietnam National Foundation for Science and Technology Development (NAFOSTED)}.
\end{center}

\vspace{-.15in}

\begin{abstract} We present a result of P. Ara which establishes that the  Unbounded Generating Number
property is a Morita invariant for unital rings.  Using this, we give necessary and sufficient conditions on a graph $E$ so that the Leavitt path algebra associated to $E$ has UGN.   We conclude by identifying the graphs for which the Leavitt path algebra is (equivalently)  directly finite; stably finite; 
Hermite; and has  cancellation of projectives.
\smallskip



%This research is supported by project the CS2015.01.28, the first
%author is supported by Vietnam National Foundation for Science and
%Technology Development (NAFOSTED).the most naive
\end{abstract}

\section{Introduction}

The concept of projective modules over rings is a generalization of the idea of a
vector space; and their structure theory, in some sense, may
be considered as a generalization of the theorem asserting the existence and
uniqueness of cardinalities of bases for vector spaces. Projective modules play an important role in
different branches of mathematics, in particular, homological algebra and algebraic
K-theory. In general ring theory it is often convenient to impose certain conditions on
the projective modules, either to exclude pathological cases or to ensure better behavior.
For  rings we have the following successively more restrictive cancellation-type conditions  
on the projective (and, in particular, the free) modules:
\begin{enumerate}
\item[(1)] Invariant Basis Number (in short: IBN),
\item[(2)] Unbounded Generating Number (in short: UGN)
\item[(3)] stable finiteness,
\item[(4)] the Hermite property (in P. M. Cohn's sense), and
\item[(5)] cancellation of projectives.
\end{enumerate}
Definitions of these properties are given below.   It is easily verified that each of these conditions is left-right symmetric and each implies the previous condition.  Moreover, in general, all these classes  are distinct.

The conditions (1) - (3) have been well studied in both algebraic and topological settings.  For basic properties of rings with these first three properties we suggest
\cite{c:srotibp}. 
% frequently among  the hypotheses in theorems about rings, both in algebra and topology, many years ago. 
For additional examples of rings satisfying property  (3), see \cite{cs:othtoctc} and the references given there.
By finding conditions for an embedding of a (non-commutative)
ring in a skew field to be possible, P. M. Cohn  discovered the theory of free
ideal rings, in which properties (1) - (5) above play an important role  (see, e.g., \cite{c:firaligr}). We  refer the reader to \cite{c:fhrtsd} and \cite{c:acfartbp}
for an investigation of rings having  (4) and (5), respectively.   It is fair to say that, in general, it is not at all easy to decide whether a given ring has any one of these properties.

Given a (row-finite) directed graph $E$ and field $K$,  Aranda Pino and the first author in \cite{ap:tlpaoag05}, and independently Ara, Moreno, and Pardo in \cite{amp:nktfga},  introduced the 
 \emph{Leavitt
path algebra}  $L_K(E)$.    These Leavitt path algebras generalize the Leavitt algebras $L_K(1, n)$ of
\cite{leav:tmtoar}, and also contain many other interesting classes of algebras.
In addition, Leavitt path algebras are intimately related to graph $C^*$-algebras
(see \cite{r:ga}). In \cite{ag:lpaosg} Ara and Goodearl introduced and investigated the {\it Cohn path algebra}  $C_K(E)$ of $E$ having  coefficients in a field $K$.
%, which is in fact a Leavitt path algebra (see \cite[Theorem 1.5.17]{aam:lpa}). 
Recently, Kanuni and the first author \cite{ak:cpahibn}
have shown that  $C_K(E)$ has IBN for every finite graph $E$.  On the other hand,  as of the writing of this article, it is  an open question to give necessary and sufficient conditions on $E$ so that $L_K(E)$ has IBN.      However,  regarding the remaining four aforementioned properties, we are able to completely classify those graphs $E$ for which   $L_K(E)$ has UGN (Theorem \ref{thm4.10}), as well as classify those graphs $E$ for which $L_K(E)$ satisfies (equivalently) properties (3), (4), and (5) (Theorem \ref{thm5.2}).      We achieve similar results for $C_K(E)$ as well.   

The article is organized as follows.   For the remainder of this introductory section we recall the germane background information.   
In Section \ref{UGNMoritasection} we present Ara's proof  that the Unbounded Generating Number property is a Morita invariant
property in the class of all unital rings (Theorem~\ref{thm 3.8}).  
%(This property had been heretofore unresolved, see e.g. \cite[Problem 5.2]{hv:ibnarpfr}.)    
In Section \ref{LpashaveUGNSection}  we give a necessary and sufficient condition for the Leavitt path
algebra of a finite source-free graph to have Unbounded Generating Number
(Theorem \ref{thm4.5}). Then, by using Theorem~\ref{thm 3.8} and the source elimination process,
we obtain a criterion for the Leavitt path algebra of an arbitrary finite graph
to have Unbounded Generating Number (Theorem \ref{thm4.10}). Consequently, we get
a criterion for the Cohn path algebra of a finite graph to have  Unbounded Generating Number
(Corollary \ref{cor4.11}).   We conclude with  Section \ref{4moreconditionssection}, in which we describe (Theorem \ref{thm5.2}, resp.  Corollary \ref{cor5.4}) those graphs $E$ for which $L_K(E)$ (resp., $C_K(E)$)  
have any one of the (equivalent) aforementioned properties (3) - (5).

\smallskip

%We complete this preliminary section by establishing notation and  recalling  some standard definitions.  
Throughout this note, all rings are nonzero, associative with identity and
all modules are unitary.   The set of nonnegative integers is denoted by $\mathbb{N}$, the positive integers by $\mathbb{N}^+$.    

A (directed) graph
$E = (E^0, E^1, s, r)$ (or simply $E = (E^0, E^1)$)
consists of two disjoint sets $E^0$ and $E^1$, called \emph{vertices} and \emph{edges}
respectively, together with two maps $s, r: E^1 \longrightarrow E^0$.  The
vertices $s(e)$ and $r(e)$ are referred to as the \emph{source} and the \emph{range}
of the edge~$e$, respectively. The graph is called \emph{row-finite} if
$|s^{-1}(v)|< \infty$ for all $v\in E^0$. All graphs in this paper will be assumed
to be row-finite. A graph $E$ is \emph{finite} if both sets $E^0$ and $E^1$ are finite.
%(or equivalently, when $E^0$ is finite, by the row-finite hypothesis).
A vertex~$v$ for which $s^{-1}(v)$ is empty is called a \emph{sink}; a vertex~$v$ for which
$r^{-1}(v)$ is empty is called a \emph{source}; a vertex~$v$ is called  \emph{isolated}
if it is both a source and a sink; and a vertex~$v$ is \emph{regular}
iff $0 < |s^{-1}(v)| < \infty$. A graph $E$ is said to be \emph{source-free} if it has
no sources.

A \emph{path} $p = e_{1} \cdots e_{n}$ in a graph $E$ is a sequence of
edges $e_{1}, \dots, e_{n}$ such that $r(e_{i}) = s(e_{i+1})$ for $i
= 1, \dots, n-1$.  In this case, we say that the path~$p$ starts at
the vertex $s(p) := s(e_{1})$ and ends at the vertex $r(p) :=
r(e_{n})$, and has \emph{length} $|p| := n$. We denote by $p^0$
the set of its vertices, that is, $p^0 = \{s(e_i) \ |\ i = 1,...,n\} \cup \{r(e_n)\}$.
%We consider the vertices in~$E^0$ to be paths of length~$0$.
If $p$ is a path in $E$, and if $v= s(p) = r(p)$, then~$p$ is a \emph{closed path
based at} $v$. A closed path $p = e_{1} \cdots e_{n}$ based at~$v$   is
a \emph{closed simple path based at}~$v$ if $s(e_i) \neq v$ for
every $i > 1$. If $p = e_{1} \cdots e_{n}$ is a closed path and all vertices
$s(e_{1}), \dots, s(e_{n})$ are distinct, then the subgraph 
$F = (F^0, F^1)$ of $E$ defined by  $F^0 =  \{s(e_{1}), \dots, s(e_{n})\}, F^1 = \{e_{1}, \dots, e_{n}\}$ 
is called a \emph{cycle}.   A graph $E$ is \emph{acyclic} if it has no cycles.    

For any graph $E= (E^0, E^1)$ and vertices $v,w\in E^0$ we write $v\geq w$ in case $v=w$ or  there exists a path $p$ in $E$ with $s(p) =v$
and $r(p) = w$.  For $v\in E^0$, the set $T(v):= \{w\in E^0\ |\
v\geq w\}$ is the \emph{tree} of $v$.  
%, that is, the set of all the vertices in the graph $E$ which follow $v$.  
(We will denote it by $T_E(v)$ when it is necessary to emphasize the
dependence on the graph $E$.) 

For any finite graph $E= (E^0, E^1)$ we denote by $A_E$ the \emph{incidence matrix} of
$E$. Formally, if $E^0 = \{v_1,..., v_n\}$, then $A_E = (a_{ij})$, the $n\times n$
matrix for which $a_{ij}$ is the number of edges in $E$ having $s(e)=v_i$ and
$r(e)=v_j$. Note that, if $v_i\in E^0$ is a sink (resp., source), then $a_{ij} = 0$ (resp., $a_{ji} = 0$) for all
$j=1,...,n$.

The notion of a Cohn path algebra 
%$C_K(E)$ of a graph $E$ with coefficients in a field $K$ 
has been defined and investigated by Ara and Goodearl \cite{ag:lpaosg}
(see also \cite{aam:lpa}). 
%For the reader's convenience we reproduce here. 
%Following \cite[Defnnition~1.5.1]{aam:lpa}, 
Specifically, for an arbitrary graph $E = (E^0,E^1,s,r)$
and an arbitrary field $K$, the \emph{Cohn path algebra}
$C_{K}(E)$ {\it of the graph}~$E$ \emph{with coefficients in}~$K$ is the $K$-algebra generated
by the sets $E^0$ and $E^1$, together with a set of variables $\{e^{*}\ |\ e\in E^1\}$,
satisfying the following relations for all $v, w\in E^0$ and $e, f\in E^1$:

%\begin{itemize}
%\smallskip
\vspace{.05in}
\qquad (1)  $v w = \delta_{v, w} w$;

\qquad (2)  $s(e) e = e = e r(e)$ and $r(e) e^* = e^* = e^*s(e)$;  \ and

\qquad (3)  $e^* f = \delta_{e, f} r(e)$.
%\smallskip
\vspace{.05in}

%\end{itemize}
Let $I$ be the ideal of $C_K(E)$ generated by all elements of the form
$v-\sum_{e\in s^{-1}(v)}ee^*$, where $v$ is a regular vertex.
Then the $K$-algebra $C_K(E)/I$ is called the \emph{Leavitt path algebra}
of $E$ with coefficients in $K$, denoted $L_K(E)$.   

Typically the Leavitt path algebra $L_K(E)$ is  defined without reference to Cohn
path algebras,  rather, it is defined  as   the $K$-algebra  generated by
the set $\{v, e, e^*\ |\ v\in E^0, e\in E^1\}$
which satisfies the  above conditions (1), (2), (3), and the additional condition:

\vspace{.05in}
\qquad (4)  $v= \sum_{e\in s^{-1}(v)}ee^*$ for any  regular vertex $v$.
\vspace{.05in}

%More precisely, in \cite{ap:tlpaoag05} Aranda Pino and the first author firstly introduced the Leavitt path algebra of a (row-finite) graph with coefficients in a field. These Leavitt path algebras generalize the Leavitt algebras $L(1, n)$ of \cite{leav:tmtoar}, and also contain many other interesting classes of algebras over fields.

%For any graph $E= (E^0, E^1)$, all elements of the form  $\{ v, e, e^{\ast} \mid v\in E^0, e\in E^1\}$ are nonzero in  $L_{K}(E)$  (see, e.g., \cite[Proposition 3.4]{tomf:lpawciacr}).
If the graph $E$ is finite, then both $C_K(E)$
and $L_K(E)$ are unital rings, each having identity $1=\sum_{v\in E^0}v$ (see, e.g.,
\cite[Lemma 1.6]{ap:tlpaoag05}).

%Furthermore,  it readily follows that
%$L_{K}(E)$ is, in fact, the ``largest'' algebra generated by the elements
%$\{v, e, e^{\ast} \mid v\in E^0, e\in E^1\}$ satisfying the above relations (1) -- (4),
%in other words, $L_{K}(E)$ has the following \textit{universal} property:
%If~$A$ is an $K$-algebra generated by a family of elements
%$\{a_{v}, b_{e}, c_{e^{\ast}} \mid v\in E^0, e\in E^1\}$
%satisfying the analogous to (1) -- (4) relations above,
%then there always exists an $K$-algebra homomorphism
%$\varphi: L_{K}(E) \rightarrow A$ given by ${\varphi(v) = a_{v}}$,
%${\varphi(e) = b_{e}}$ and ${\varphi(e^{\ast}) = c_{e^{\ast}}}$.

%%%%%%%%%%%%%%%%%
%%%%%%%%%%%%%%%%%
%%%%%%   UGN Section
%%%%%%%%%%%%%%%%%
%%%%%%%%%%%%%%%%%

\section{Rings having Unbounded Generating Number}\label{UGNMoritasection}
The goal of this section is  to show that the UGN property is a Morita  invariant 
 for unital  rings.

 For many fundamental rings $R$ (e.g., fields and $\mathbb{Z}$), it is well-known  that any two bases for a free right $R$-module necessarily contain the same number of elements;  rephrased, if $R^m \cong R^n$ as right $R$-modules, then $m=n$.  Such rings are said to have the Invariant Basis Number (IBN) property.   On the other hand, in  fundamental work done by  W.G. Leavitt, it is shown (among other things) that, for any pair $(n,N)$ of positive integers with $N>n$, and any field $K$, there exists a $K$-algebra $L_K(n,N)$ for which $R^n \cong R^{N}$.   Germane in this context is the observation that for the graph $R_N$ consisting of  one vertex and $N$ loops, the Leavitt path algebra $L_K(R_N)$ is isomorphic to $L_K(1,N)$.  Additional examples of Leavitt path algebras which lack the IBN property abound.   
 %However, as previously mentioned, there is currently no characterization of the graphs $E$ for which $L_K(E)$ has (or lacks) the IBN property.     
 Appropriate in this context is the observation that rings which lack the IBN property fail to have a ``cancellation of projectives":   specifically, if $R^n \cong R^N$ with $n<N$, then $R^n \oplus R^{N-n} \cong R^N \cong R^n \cong R^n \oplus \{0\}$, but obviously $R^{N-n} \not\cong \{0\}$. 
 
   %\cite{leav:tmtoar}. A ring for which the number of elements in a basis for any
%free module is uniquely determined, is called \emph{Invariant Basis Number}
%(for short, \emph{IBN}). In other words, a ring $R$ has \emph{IBN} if, for any pair of
%positive integers $m$ and $n$, $R^m\cong R^n$ as right $R$-modules implies $m=n$. %For basic properties of IBN rings
%we may refer to \cite{l:lomar}, \cite{hv:ibnarpfr} and \cite{c:firaligr}, for example.
%There are many examples of rings that have IBN, for example, commutative rings,
%one-side Noerherian rings and in particular, Cohn path algebras of finite graphs \cite{ak:cpahibn},
%but we shall meet examples of rings no satisfying this property, for example,
%the Leavitt algebras $L(1, n)$ \cite{leav:tmtoar}.

There are  natural, well-studied  ``cancellation-type" properties of projective modules over general rings which are stronger than the IBN property.    
%It is, in general, not at all easy to decide whether a given ring has IBN. In order to understand IBN more thoroughly, and come up with more classes of rings with IBN, it is advantageous to consider other, somewhat stronger, conditions. In this section, we will consider rings with the Unbounded Generating Number property.

\begin{defn}
A ring $R$ is said to have \emph{Unbounded Generating Number} (\emph{UGN} for short) if,
for each positive integer $m$, any set of generators for the free right $R$-module $R^m$ has cardinality $\geq m$.  \hfill $\Box$ 
\end{defn}

Another terminology which has been used for the UGN property is the  ``rank condition" (see, e.g., \cite{hv:ibnarpfr} and \cite[Section 1C]{l:lomar}).    
We note the following easily verified equivalent characterizations of the UGN property.
%; these appear for instance in   \cite{c:srotibp}, \cite{l:lomar} and \cite{c:firaligr}.  

\begin{rem}\label{rem3.2}
The following conditions are equivalent for any ring $R$:

(1) $R$ has Unbounded Generating Number;

%(2) For any positive integer $m$ there is a finitely generated right $R$-module which cannot be generated by $m$ elements;

%(3) For any pair of positive integers $m$ and $n$, if there is an epimorphism of free right $R$-modules: $R^n \longrightarrow R^m$, then $n\geq m$;

(2) For any pair of positive integers $m$ and $n$, and any right $R$-module $K,$
$R^n \cong R^m \oplus K$ implies that $n\geq m$;

(3) For any $A\in M_{m\times n}(R)$ and $B\in M_{n\times m}(R)$, if $AB= I_m$, then
$n\geq m$.    \hfill $\Box$ 
\end{rem}

\begin{rem}\label{UGNimpliesIBN}
By condition (3) in  Remark~\ref{rem3.2} we see that  the UGN property is indeed a left-right symmetric
condition in general.   
 Moreover, using  condition (2),   it is clear that if $R$ is UGN, then necessarily $R$ is IBN.  We will show in Example \ref{IBNnotUGN}  that the converse is not true, even in the context of Leavitt path algebras.      \hfill $\Box$ 
\end{rem}

\begin{lem}[{cf.~\cite[Proposition 2.4]{c:srotibp}}]\label{lem 3.3} Let $f: R\longrightarrow S$ be
a unital ring homomorphism. If $S$ has Unbounded Generating Number, then so too does $R$.
\end{lem}
\begin{proof} If $A\in M_{m\times n}(R)$ and $B\in M_{n\times m}(R)$ are matrices
for which $AB= I_m$, then we get an analogous equation in matrices over $S$
by applying the homomorphism $f$ entrywise, so $n\geq m$ by the UGN property on $S$.
%Alternatively, we can also prove the desired result by applying the functor $- \otimes_RS$ to free right $R$-modules.
\end{proof}

There are many classes of rings which have Unbounded Generating Number.     For example, any field easily has UGN.  More generally, using Lemma \ref{lem 3.3}, any commutative ring $R$ has UGN:  pick a maximal ideal $M$ of $R$, and consider the natural surjection from $R$ to the field $R/M$.  Additionally, any Hopfian ring $R$ (a ring for which every right module epimorphism $\varphi: R^n \to R^n$ is an isomorphism for each  $n\in \mathbb{N}^+$) is UGN; these include the Noetherian rings and self-injective rings.

The rest of this section is taken up in showing that the UGN property is a Morita invariant
for unital rings. Before doing so, we recall some fundamental concepts, as well as
establish some useful facts.

\begin{defn}\label{monoiddefs}
Let $M$ be an abelian monoid (i.e., $M$ is a set, and $+$ is an associative commutative binary operation on $M,$ with a neutral element).   

(1)  We define a relation $\leq$ on $M$ by setting
$$x\leq y \ \ \mbox{in case there exists }  z\in M  \ \mbox{for which} \ x+ z = y.$$
Then $\leq$ is a preorder (reflexive and transitive).  

(2)   We call an
element $u\in M$ \emph{properly infinite} if $2u \leq u$ in $M.$ It is easy to check
that if $u\leq v$ and $v\leq u$ in $M$, and $u$ is properly infinite, then $v$ is also
properly infinite.

(3) An element $d\in M$ is called an \emph{order-unit} if, for any $x\in M,$  there
exist a positive integer $n$ such that $x\leq nd$.

(4) An order-unit $d\in M$ is said to have \emph{Unbounded Generating Number}
(for short, \emph{UGN}) if, for every pair of positive integers $n, n'$, if
$nd\leq n'd$ in $M,$ then $n\leq n'$.    \hfill $\Box$ 
\end{defn}

For any ring $R$ we denote by $\mathcal{V}(R)$ the set of isomorphism classes
(denoted by $[P]$) of finitely generated projective right $R$-modules, and we endow
$\mathcal{V}(R)$ with the structure of an abelian monoid by imposing the operation:
$$[P] + [Q] = [P\oplus Q]$$ for any isomorphism classes $[P]$ and $[Q]$. By Remark \ref{rem3.2}(2), we  see that the ring $R$ has UGN if and only if the order-unit $[R]$ of $\mathcal{V}(R)$ has UGN.

\begin{lem}\label{lem 3.6}
Let $M$ be an abelian monoid and $\mu \in M.$ Then $\mu$ does not have Unbounded Generating
Number if and only if $n\mu$ is properly infinite for some positive integer $n$.
\end{lem}
\begin{proof}   Assume that $\mu$ does not have UGN, that is, there
exist two positive integers $m, n$ such that $m > n$ and $m\mu +x = n\mu$ for some $x\in M.$
Set $k:= m- n>0$. We then have that $$m\mu = n\mu + k\mu = (m\mu + x) + k\mu = m\mu + (k\mu + x),$$  
which by substituting gives $m\mu = (m\mu + (k\mu + x)) + (k\mu + x) = m\mu + 2k\mu + 2x$, which by an easy
induction gives 
$m\mu = (m+ tk)\mu + tx $ for all $t\in \mathbb{N}$.  Then adding $x$ to both sides yields $n\mu = m\mu + x = (m+ tk)\mu + (t+1)x$ for all  $t\in \mathbb{N}^+$.  In particular, $(m+tk)\mu \leq n\mu$ for all  $t\in \mathbb{N}^+$.  So, without loss of
generality, we may assume that $m\geq 2n$. But then \[2n\mu + (m-2n)\mu +x =
m\mu +x = n\mu,\] that is, $2n\mu\leq n\mu$. Therefore, $n\mu$ is properly infinite.

The converse is obvious.  
\end{proof}

\begin{prop}\label{prop 3.7}
Let $M$ be an abelian monoid.  Let  $d_1$ and $d_2$ be  order-units in $M$. Then 
$d_1$ has Unbounded Generating Number if and only if so does $d_2$.
\end{prop}
\begin{proof} Assume that $d_1$ does not have UGN; we show that the same holds for $d_2$ as well. By Lemma~\ref{lem 3.6},
there exists a positive integer $n$ such that $nd_1$ is properly infinite,
i.e., $2nd_1\leq nd_1$. Since $d_2$ is an order-unit in $M,$  there exists a positive integer
$\ell$ such that $u := nd_1 \leq \ell d_2=: v$. Furthermore, as $d_1$ is an order-unit in $M$,
there exists a positive integer $k$ such that $v\leq kd_1$.

We show now that $v\leq u$. If $k\leq 2n$, then we have that $v\leq kd_1 \leq
2nd_1\leq nd_1 = u$. Otherwise, let  $t$ be the minimum positive integer for which $0< k-tn\leq 2n$.
But then
$$ kd_1   =   2nd_1 + (k-2n)d_1  \leq  nd_1 + (k-2n)d_1  =   (k-n)d_1, $$
which similarly gives 
$$ (k-n)d_1   =   2nd_1 + (k-3n)d_1  \leq  nd_1 + (k-3n)d_1   =   (k-2n)d_1, $$
which then  by induction and the transitivity of $\leq$ gives
$ kd_1  \leq (k-tn)d_1$.   But then we have
$v \leq kd_1 \leq (k-tn)d_1  \leq 2nd_1 \leq nd_1 = u.$

 % $$2nd_1 + (k-3n)d_1 \ \ \leq \ \ nd_1 + (k-3n)d_1  \ \ = \ \  (k-2n)d_1, $$

%\bigskip

%\begin{equation*}
%\begin{array}{rl}
% kd_1 & =  \ \ 2nd_1 + (k-2n)d_1 \ \ \leq \ \ nd_1 + (k-2n)d_1  \ \ = \ \  (k-n)d_1\\
%&= 2nd_1 + (k-3n)d_1\leq nd_1 + (k-3n)d_1 = (k-2n)d_1\\
%&. . . . . . . . . . . . .\\
%& = 2nd_1 + (k-(t+1)n)d_1\leq nd_1 + (k-(t+1)n)d_1=\\
%& = (k-tn)d_1\leq 2nd_1 \leq nd_1= u.
%\end{array}
%\end{equation*}

So we have $u\leq v$ and $v\leq u$.   From these observations and the assumption that $u$ is properly infinite, we conclude by the  observation made in Definition \ref{monoiddefs}(2)  that $v = \ell d_2$ is also properly
infinite. Therefore, $d_2$ does not have UGN, by  Lemma~\ref{lem 3.6}.
\end{proof}

It is known that the IBN property is not a Morita invariant property for rings (see, e.g.,
\cite[Exercise 11, page 502]{l:lomar}; such examples where both of the  rings are  Leavitt
path algebras can be constructed as well).   
In contrast, we now present the main result of this section.  

% in the context of Leavitt path algebras ); an example in the context of Leavitt path algebras can be given . However, as a corollary of Proposition~\ref{prop 3.7}, the following theorem, showing the UGN property is a Morita invariant property for rings, also solves a part of Problem 5.2 left open in \cite{hv:ibnarpfr}.
\begin{thm}\label{thm 3.8}
Let $R$ and $S$ be Morita equivalent unital rings. Then $R$ and $S$ have Unbounded Generating
Number simultaneously.
\end{thm}
\begin{proof} 
Let $\Phi : Mod-R  \rightarrow Mod-S$ be the presumed equivalence of categories.   Then the restriction $\varphi = \Phi |_{\mathcal{V}(R)} : \mathcal{V}(R) \rightarrow \mathcal{V}(S)$ is a monoid isomorphism.   Since a monoid isomorphism clearly takes order-units to order-units, and  the UGN property of an element in a monoid is  a monoid-isomorphism invariant, we see that if $[R]$ has UGN in $\mathcal{V}(R)$, then  $\varphi([R])$ is an order-unit in $\mathcal{V}(S)$ having  UGN.   But then by Proposition \ref{prop 3.7} we get that the order-unit $[S]$ of $\mathcal{V}(S)$   has UGN as well.   
%Assume that $R$ has Unbounded Generating Number. Equivalently, the order
%unit $[R]$ of $\mathcal{V}(R)$ has UGN. Since $R$ is Morita equivalent to $S$, there
%exists a finitely generated projective generator $P$ in the category of right modules
%$\mathcal{M}_R$ such that $S\cong End_R(P)$. By $P$ is a generator, there exists a
%positive integer $n$ such that the regular right $R$-module $R$ is isomorphic to a
%direct summand of $P^n$. Then $[R]\leq n[P]$ in $\mathcal{V}(R)$, from which it is
%clear that $[P]$ is an order-unit in $\mathcal{V}(R)$. Applying Proposition~\ref{prop 3.7},
%we immediately get that $[P]$ has UGN.
%The functor $- \otimes_S P: \mathcal{M}_S\longrightarrow \mathcal{M}_R$ is a
%category equivalence. Consequently, there is an isomorphism $\varphi: \mathcal{V}(S)
%\longrightarrow \mathcal{V}(R)$ of pre-ordered abelian monoids, given by the rule
%$\varphi([A]) = [A\otimes_SP]$. Observing that $\varphi([S]) = [P]$, we conclude
%that the order-unit $[S]$ of $\mathcal{V}(S)$ has UGN, that is, $S$ has UGN.
\end{proof}

\begin{rem}\label{historyofUGNresult}
Using the {\it separative} property of $\mathcal{V}(L_K(E))$ established in \cite{amp:nktfga} (for any finite graph $E$), we  had been able to fairly easily verify that the UGN property is a Morita invariant within the class of unital Leavitt path algebras.  (This property was of sufficient strength to allow us to use it  to achieve  the original proof of our main result, Theorem \ref{thm4.10}.)    Subsequently, when informed about this property of Leavitt path algebras,  P. Ara realized that such Morita invariance indeed holds for {\it all} unital rings.  We thank him for allowing us to use his proof of this more general property in our article; it has been  presented here as Lemma \ref{lem 3.6}, Proposition \ref{prop 3.7}, and Theorem \ref{thm 3.8}.   We note that Theorem \ref{thm 3.8} answers \cite[Problem 5.2]{hv:ibnarpfr}.\footnote{It is interesting to note also that the following question appears as an Exercise  in Section 0.1 of Cohn's book  \cite{c:firaligr}:     

  \ \ \ \ \ \ \ \ 9$^*$.  {\it Which of IBN, UGN, and weak finiteness (if any) are Morita invariants?}   \\ We know of no place in the literature where a solution to the UGN portion of the question appears.  The asterisk ${}^*$ indicates that Cohn viewed this as a ``harder" question;  however, it was not considered an ``open" question (which would have instead merited a ${}^\circ$ designation).}       \hfill $\Box$ 
\end{rem}

%%%%%%%%%%%%%%%%%%%%%%
%%%%%%%%%%%%%%%%%%%%%%
%
%   Section describing Lpas which are UGN
%
%%%%%%%%%%%%%%%%%%%%%%
%%%%%%%%%%%%%%%%%%%%%%

\section{Leavitt path algebras having Unbounded Generating Number}\label{LpashaveUGNSection}
In this section we establish the main result of the article, to wit, we  give necessary and sufficient conditions for the Leavitt path algebra $L_K(E)$  of
a finite graph $E$ with coefficients in a field $K$  to have  Unbounded Generating Number.

%We begin by recalling the description of the monoid of isomorphism classes of finitely generated projective modules of Leavitt path algebras which is due to Ara, Moreno and Pardo \cite{amp:nktfga}. Namely, 
Following \cite{amp:nktfga}, for any directed graph
$E=(E^0, E^1, s, r)$ we define the monoid $M_E$ as follows. 

\begin{defn}\label{Tdefn}
We denote by $Y_E$ (or simply by $Y$, if the graph $E$ is clear) the free abelian
monoid (written additively) with generators $E^0$. Relations are defined on $Y_E$ by setting

\medskip

$\hspace{1.75in}  v = \sum_{e\in s^{-1}(v)}r(e)  \hfill ({\rm M})$ 

\medskip
\noindent
for every regular vertex $v\in E^0$.
Let $\sim_{E}$ be the congruence relation on $Y_E$ generated by these relations.
Then $M_E$ is defined to be the  monoid $ Y_E/{\sim_E}$.   The elements of $M_E$ are usually denoted by $[x]$,  for $x\in Y_E$.     \hfill $\Box$ 

\end{defn}

In the literature the generators of $Y$ are sometimes denoted $\{a_v \ | \ v\in E^0\}$ (rather than by $E^0$ itself)  to indicate that $M_E$ is not being viewed as any sort of quotient of elements of $L_K(E)$; we have chosen to use the less cumbersome of the two notations.    Alternatively, $Y$ may be viewed as $\mathbb{N}^{|E^0|}$, where $E^0 = \{v_1, v_2, \dots, v_n\}$, and $v_i$ is associated with the $i^{th}$ standard basis vector in  $\mathbb{N}^{|E^0|}$ for each $1\leq i \leq n$.  
\begin{exas}\label{monoidexamples}  We identify the monoid $M_E$ for some important classes of graphs. 

\medskip

(1)    For each $n\in \mathbb{N}^+$, let $A_n =  \xymatrix{   \bullet^{v_1} \ar[r] & \bullet^{v_2} \ar[r] & \cdots \  \   \bullet^{v_{n-1}} \ar[r] & \bullet^{v_{n}}  } $.   Then in $M_{A_n}$ we have $[v_1] = [v_2] = \cdots = [v_n]$, and $M_{A_n} = \{j [v_n] \ | \ j \in \mathbb{N}\} \cong \mathbb{N}.$

\medskip

(2)    For each $n\in \mathbb{N}^+$ let $C_n$ be the ``single cycle graph of $n$ vertices", with vertices labelled $v_1, v_2, \dots, v_n$.     Then in $M_{C_n}$ we have, as in the previous example,  $[v_1] = \cdots = [v_n]$, and $M_{C_n} = \{j [v_n] \ | \ j \in \mathbb{N}\} \cong \mathbb{N}.$ 

\medskip

(3)  For each integer  $n \geq 2$  let

 $$R_n = 
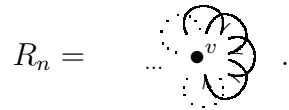
\xymatrix{ & {\bullet^v} \ar@(ur,dr)  \ar@(u,r) \ar@(ul,ur)  \ar@{.} @(l,u) \ar@{.} @(dr,dl)
\ar@(r,d)  \ar@{}[l] ^{\hdots} } \ \ \ \ \  .$$   

\bigskip
\noindent
($R_n$ is the ``rose with $n$ petals" graph; it is central to the theory of Leavitt path algebras, as $L_K(R_n) \cong L_K(1,n)$, the aforementioned Leavitt algebra of order $n$.)   Then $M_{R_n} = \{0,1[v],2[v],\dots,(n-1)[v]\}$, where $n[v] = [v]$.  

\medskip

(4)   The Toeplitz graph is the graph  $$ \mathcal{T} = \ \ \  \ \  
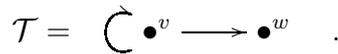
\xymatrix{  \bullet^v   \ar@(dl,ul)  \ar[r]        &\bullet^w  } \ \ \ .$$   Then $M_\mathcal{T} = \{n[v] + n'[w] \ | \ n,n'\in  \mathbb{N}$, and $[v] = [v+w]\}$.   \hfill $\Box$

\end{exas}

 In \cite{amp:nktfga} Ara, Moreno and Pardo establish the following fundamental result.

\begin{thm}[{\cite[Theorem 3.5]{amp:nktfga}}]\label{thm4.1}
Let $E = (E^0, E^1)$ be a row-finite graph and $K$ any field. Then the map
$[v]\longmapsto [vL_{K}(E)]$ yields an isomorphism of abelian monoids $M_E\cong \mathcal{V}(L_{K}(E))$.
In particular, under this isomorphism, we have $[\sum_{v\in E^0}v]\longmapsto [L_{K}(E)]$.
\end{thm}

Applying Theorem~\ref{thm4.1} and Remark~\ref{rem3.2}(2), we immediately get the following corollary, which
provides us with a criterion to check the UGN property of  $L_K(E)$  in terms of the monoid $M_E$.

\begin{cor}\label{cor4.2}
Let $E = (E^0, E^1)$ be a finite graph and $K$ any field.
Then the following  are equivalent:

(1) $L_K(E)$ has Unbounded Generating Number.

%\item[(2)] For any pair of positive integers $m, n$ and $[P]\in \mathcal{V}(L_K(E))$,
%if $\ m[L_K(E)] + [P]=n[L_K(E)]$ in $\mathcal{V}(L_K(E))$, then $m\leq n$;

(2) For any pair of positive integers $m$ and $n$, and any $[x]\in M_E$,
$$\mbox{if} \ \ m[\Sigma_{v\in E^0}v] +\ [x] = n[\Sigma_{v\in E^0}v] \ \mbox{ in } M_E, \mbox{ then } m\leq n.$$
%\end{enumerate}
\end{cor}

As usual,   $(\ )^t$ notation denotes the standard  transpose of a matrix.  Also, for  matrices $A=(a_{ij})$
and $B=(b_{ij})\in M_{m\times n}(\mathbb{Z})$, $A\leq B$ means  that $a_{ij}\leq
b_{ij}$ for all $i = 1,..., m$ and $j = 1, ..., n$.     

\begin{defn}\label{sourcecycledef}
Let $c$ be a cycle in the graph $E$.   We call $c$ a {\it source cycle} in case $|r^{-1}(v)| = 1$ for all $v\in c^0.$    \hfill $\Box$ 
\end{defn}

We will utilize heavily the following graph-theoretic result.

\begin{lem}\label{graphtheorylemma}
Let $E$ be a finite source-free graph  for which no cycle is a source cycle.   
%each cycle $c$ in $E$, there exists some $v \in c^0$ for which  $|r^{-1}(v)|>1$.    \text{ for some } v \in c^0.$$} 
Then there exists a vertex $v \in E^0$ for which there are two distinct cycles based at $v$, and for which  $|r^{-1}(v)|\geq 2$.
\end{lem}

\begin{proof}
Let $v_1\in E^0$ be an arbitrary vertex.  Then, as $v_1$ is not a source, there exists   $e_1\in
E^1$ such that $r(e_1) = v_1$. Set $v_2:=s(e_1)$. If $v_2 = v_1$, then we get a cycle $c = e_1$.
Otherwise, as $v_2$ is not a source,  there exists $e_2\in E^1$
such that $r(e_2) = v_2$. Let $v_3:=s(e_2)$, and we continue to repeat this process. 
Since $E$ is finite, there exists a  smallest integer $n$
such that $r(e_{n}) = v_{n}$ and $s(e_n) = v_i$ for some $i$ ($1\leq i \leq n$), and
$c:=e_n \cdots e_{i+1}e_{i}$ is a cycle in $E$.

Using the ``no source cycle" hypothesis  on  $c$, there then exists  $j \in \{i, ..., n\}$ such that $|r^{-1}(v_j)| \geq 2$. Let
$f_1\in r^{-1}(v_j)$ such that $f_1 \neq e_j$. If $s(f_1) = v_j$, then $v_j$ is the base of
two distinct cycles $c$ and $c':= f_1$, as desired. Otherwise, since $w_1:= s(f_1)$ is not a
source, there exists  $f_2\in E^1$ such that $r(f_2) = w_1$. Since $E$ is finite, we must eventually arrive at one of these two cases:

\smallskip

\underline{Case 1}. There exists an integer $m$ such that $r(f_m) = w_{m-1}$ and $s(f_m) = v_k$
for some $k$ $(1\leq k\leq n)$. Choose the smallest such $m$.  Then, we have that
\begin{equation*}
c' = \left\{
\begin{array}{lcl}
e_{j-1}\cdots e_kf_{m} \cdots f_1&  & \text{if }k< j, \\
&  &  \\
f_{m}\cdots f_1&  & \text{if }k = j,\\
&  &  \\
e_{j-1}\cdots e_ie_{n}e_{n-1}\cdots e_kf_{m}\cdots f_1&  & \text{if }k> j%
\end{array}%
\right.
\end{equation*}%
is a cycle based at $v_j$, which is different from $c$, for which $|r^{-1}(v_j)| \geq 2$, as desired.

\smallskip

\underline{Case 2}. There exists an integer $m$ such that $r(f_m) = w_{m-1}$ and $s(f_m) = w_\ell$
for some $\ell$ $(1\leq \ell \leq m-1)$.  Choose the smallest such $m$.   Then 
%\begin{equation*} 
$c_1:=f_{m}\cdots f_{\ell}$
%\end{equation*}
 is a cycle in $E$.     In this case, we repeat the process described above, starting with the cycle $c_1$.   
 
 \smallskip
 
 In this way we produce a sequence of cycles $c, c_1, \dots, c_t$.   If for some $c_i$ we are in Case 1, then we are done.    We note that if we start the process with some $c_i$, and one of the vertices appearing in the process for $c_i$ is a vertex which has previously appeared in the process corresponding to one of the cycles $c, c_1, \dots, c_{i-1}$, then we may find a vertex of the desired type by constructing two cycles in a manner similar to that done in Case 1.     Therefore, as $E$ is finite, we must eventually reach Case 1, thus completing the proof.  
\end{proof}

The existence of a vertex of the type described in Lemma \ref{graphtheorylemma} will play a key role in the following result.

\begin{lem}\label{lem 4.3}
Let $E=(E^0,E^1,r,s)$ be a finite source-free graph in which no cycle is a source cycle.   Let $n:= |E^0|$.  Then for each positive integer $a$ there exists a row vector
$\overrightarrow{m_a}=[m_1\ ...\ m_n]\in \mathbb{N}^n$ 
such that $m_i\geq a$ for all $i = 1,..., n$, and
$$(A_E^t-I_n)\overrightarrow{m_a}^t\geq [a\ ...\ a]^t.$$
\end{lem}
\begin{proof} 
By Lemma \ref{graphtheorylemma} there exists a vertex $w_1\in E^0$ such that
$w_1$ is a base of distinct cycles,   and $|r^{-1}(w_1)|\geq 2$.
If $E^0\setminus T_E(w_1)\neq \emptyset$, then we consider the subgraph
$F=(F^0,F^1,r|_{F^1},s|_{F^1})$, where $F^0 :=E^0\setminus T(w_1)$ and $F^1:=r^{-1}(F^0)$.
Note that we always have  $r^{-1}_F(v) = r^{-1}_E(v)$ for any vertex $v\in F^0$.
This implies that $F$ is a source-free graph in which no cycle is a source cycle.
%such that  each cycle $c$ in $F$ having $|r^{-1}_F(v)|>1 \text{ for some } v \in c^0.$

So we may apply  Lemma \ref{graphtheorylemma} to $F$, to conclude the existence of  a vertex
$w_2 \in F^0$ such that $w_2$ is a base of distinct cycles  in $F$ and $|r^{-1}_F(w_2)|\geq 2$. If $F^0\setminus T_F(w_2)\neq
\emptyset$, we continue to repeat the process. Since $E$ is
finite, this process ends after finitely many (say, $k$) steps.   We consider the set of  vertices $$\{w_1,w_2, ..., w_k\}.$$ Anticipating an induction argument, we note that the number of steps required to complete the same process starting with either of the two graphs $E^0 \setminus T(w_1)$ or $T(w_1)$ is less than $k$.

We use induction on $k$ to establish the result.    If $k=1$ we  have that $E^0 = T_E(w_1)$. 
By renumbering vertices in $E^0$, without loss of generality, we may assume that
$$E^0=\{v_1,v_2,...,v_n\}, v_1 :=  w_1, \ \mbox{and} \ |r^{-1}(v_1)|\geq 2.$$

%\bigskip

%{\bf Why must the $|E^0| = 1$ case be separate?}

%\bigskip 

%If $|E^0| = 1$, then $E$ is the graph $R_n$ of Examples \ref{monoidexamples}, where necessarily $n\geq 2$ as no cycle in $E$  is a source cycle.      Then $A_E^t - I_1 = (n-1)$.  Now pick $a\in \mathbb{N}^+$; then choosing  $m_1=a$ we get $(A_E^t - I_1)(m_1) = ((n-1)a) \geq (a)$.   So the result is true in this case.

%\bigskip

%So we may assume that $|E^0| \geq 2.$

%\bigskip

For each positive integer $a$, we choose  the row vector $\overrightarrow{m_a}:=
[m_1\ ...\ m_n]$ according to the following algorithm.   

- Define $m_1:=3na$.

- For any $i \in \{2, ..., n\}$, since $E^0 = T_E(v_1)$ we have $v_i\in T_E(v_1)$, and hence, 
there exists a path $p$ (which can be chosen of  minimal length) such that $s(p) = v_1$ and $r(p) = v_i$.
Then  we define  $$m_i := m_1 - |p|a = (3n -|p|)a, $$ 
where we denote by $|p|$ the length of
the path $p$.  

 We note that since $|E^0| = n$, we always have that $1\leq|p|\leq n$,  so that $m_i\geq 2na$. Also, for any
$j \in \{2, ..., n\}$, $v_j \in r(s^{-1}(v_i))$ for some  vertex $v_i\neq v_j$, so there exists 
an $i\in \{1, ..., n\}$ such that $m_j = m_i -a$.  
%, by choosing $m_i$'s.)

%For each positive integer $a$, we choose  the row vector $\overrightarrow{m_a}:=
%[m_1\ ...\ m_n]$ by rule:

%- Let $m_1=3na$.

%- Assume that $m_1,..., m_i$ are chosen. We then will choose $m_{i+1}$ as
%follows: If there exist a vertex $v_k$ such that $k\in \{1,..., i\}$ and
%$v_{i+1}\in r(s^{-1}(v_k))$, we choose \[m_{i +1}=m_k-a.\]
%Noting that by choosing $m_i$'s, we must have that
%$$m_i\geq 3na -na =2na$$
%for all $i = 1,..., n$.

We will prove that the vector $\overrightarrow{m_a}$ satisfies the statement, in other words, that 
%By the note above, we only need to check that 
$(A_E^t-I_n)
\overrightarrow{m_a}^t\geq [a\ ...\ a]^t.$ Equivalently, we show 
\begin{center}
$a_{1j}m_1+ ... +(a_{jj}-1)m_j+...+a_{nj}m_n\geq a$ for all $j=1,...,n$.
\end{center}

\underline{For $j\in \{2,...,n\}$}: as $v_j\in T(v_1)$, there exists a vertex $v_i$
such that $v_j \in r(s^{-1}(v_i))$ and $v_i \neq v_j$.   As noted above, we then may find
an element $i\in \{1,...,n\}$ such that $a_{ij}\geq 1$ and $m_j = m_i -a$.
This implies that
\begin{equation*}
\begin{array}{rl}
a_{1j}m_1+...+(a_{jj}-1)m_j+...+a_{nj}m_n&\geq a_{ij}m_i+(a_{jj}-1)m_j \\
&\geq m_i-m_j=a.
\end{array}
\end{equation*}

\underline{For $j=1$}: as $|r^{-1}(v_1)|\geq 2$, there exist two distinct elements
$k, h\in \{1,..., n\}$ such that $a_{k1}\geq 1$ and $a_{h1}\geq 1$.
If $k\geq 2$ and $h\geq 2$, we have that
\begin{center}
$\begin{array}{rl}
(a_{11}-1)m_1+...+a_{n1}m_n&\geq -m_1+m_k+m_h\\
&\geq -3na + 2na+ 2na = a.\\
\end{array}$
\end{center}
Otherwise, without loss of generality, we may assume that $h\geq 2$.
We then have that $a_{11}-1\geq 0$ and
%\begin{center}
%$\begin{array}{rl}
$$(a_{11}-1)m_1+...+ a_{h1}m_h +...+ a_{n1}m_n \geq m_h\geq 2a.$$
%\end{array}$
%\end{center}

These two cases establish the claim.  Now we proceed inductively. For $k>1$, let $F=(F^0, F^1, r|_{F^1}, s|_{F^1})$
and  $G=(G^0, G^1, r|_{G^1}, s|_{G^1})$ be the subgraphs of
$E$ defined by: $$F^0 :=E^0\setminus T(w_1) \text{ and }F^1:=r^{-1}(F^0)$$
and $$G^0 :=T(w_1) \text{ and } G^1 := \{f\in E^1\ |\ s(f),\ r(f) \in G^0\}.$$
Clearly, $F$ and $G$ satisfy the same conditions as the graph $E$.
Then, by the induction hypothesis, for each  positive integer $a$, there exist
row vectors $\overrightarrow{x_a}=[m_1\ ...\ m_f]\in \mathbb{N}^f$ ($m_i\geq a$) and
$\overrightarrow{y_a}= [m'_1\ ...\ m'_g]\in \mathbb{N}^g$ ($m'_j\geq a$)
such that $$(A_F^t-I_f)\overrightarrow{x_a}^t\geq [a\ ...\ a]^t$$ and
$$(A_G^t-I_g)\overrightarrow{y_a}^t\geq [a\ ...\ a]^t,$$ where $f=|F^0|$,
$g = |G^0|$, and $A_F$ and $A_G$ are the incidence matrices of $F$ and
$G$, respectively.

We write the matrix $(A_E^t-I)$ of the form:
\begin{equation*}A_E^t-I=\left(\begin{tabular}{cc}
$A_F^t-I_f$&$A_{21}$\\
$A_{12}$&$A_G^t-I_g$
\end{tabular}\right)
\end{equation*}
where $A_{12}$ and $A_{21}$ are the appropriately sized rectangular submatrices of $A_E^t-I_n$, each having only nonnegative integer entries (since none of these entries is on the main diagonal of $A_E^t-I_n$).  Let
$\overrightarrow{m_a}:=[m_1\ ...\ m_f\ \ m'_1\ ...\ m'_g ]\in \mathbb{N}^n$.
We then get
$$(A_E^t-I_n)\overrightarrow{m_a}^t=\left(\begin{tabular}{cc}
$A_F^t-I_f$&$A_{21}$\\
$A_{12}$&$A_G^t-I_g$
\end{tabular}\right)\left(\begin{tabular}{c}$\overrightarrow{x_a}^t$\\
$\overrightarrow{y_a}^t$\end{tabular}\right)\geq [a\ ...\ a]^t,$$  which ends the proof.
\end{proof}

For clarification, we illustrate the ideas which arise in the proof of Lemma~\ref{lem 4.3} by presenting the following example.

\begin{example}\label{vectorprocessexample}

Let $E$ be the  graph 

$$
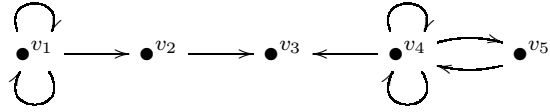
\xymatrix{ \bullet^{v_1} \ar@(ul,ur) \ar@(dr,dl) \ar[r] & \bullet^{v_2} \ar[r] & \bullet^{v_3} & \bullet^{v_4} \ar[l]     \ar@(ul,ur) \ar@(dr,dl)   \ar@/^.5pc/[r]           &\bullet^{v_5} \ar@/^.5pc/[l]  }$$

\bigskip

%\begin{center}
%            \includegraphics{graph-1.mps}
%\end{center}
\noindent
Note that $v_4$ is the base of two distinct cycles, and $|r^{-1}(v_4)| \geq 2$.   We designate $w_1 = v_4$.  
%Let $G=(G^0, G^1, r|_{G^1}, s|_{G^1})$ and $F=(F^0, F^1, r|_{F^1},
%s|_{F^1})$ be the subgraphs of $E$ defined by: $$G^0 :=T(v_4) \text{
%and } G^1 := \{f\in E^1\ |\ s(f),\ r(f) \in G^0\}$$ and
%$$F^0 :=E^0\setminus T(v_4) \text{ and }F^1:=r^{-1}(F^0).$$
Let $G$ denote the subgraph $T(w_1) = T(v_4)$, and let $F$ denote $E \setminus G = E \setminus T(w_1)$.  
Then, $F$ and $G$ are the following graphs:

$$ \ F = 
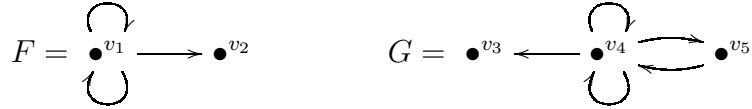
\xymatrix{ \bullet^{v_1} \ar@(ul,ur) \ar@(dr,dl) \ar[r] & \bullet^{v_2}    } \  \ \ \ \ \ \ \ \ \ \ \ G = \xymatrix{  \bullet^{v_3} & \bullet^{v_4} \ar[l]     \ar@(ul,ur) \ar@(dr,dl)   \ar@/^.5pc/[r]           &\bullet^{v_5} \ar@/^.5pc/[l]  } \ \ $$

\bigskip

%\begin{center}            \includegraphics{graph-3.mps}
%\end{center}
%and
%\begin{center}            \includegraphics{graph-2.mps}
%\end{center}

\noindent
Let $a$ be an arbitrary positive integer. As shown in the proof of Lemma~\ref{lem 4.3},
we  choose vectors
$\overrightarrow{x_a}=[m_1\ \ m_2]\in \mathbb{N}^2$ and
$\overrightarrow{y_a}= [m_3\ \ m_4\ \ m_5]\in \mathbb{N}^3$ as follows:
\[m_1 = 3a|F^0|= 6a\ \ \text{ and }\ \ m_2 = m_1 - a = 5a,\]
\[m_4 = 3a|G^0|= 9a\ \text{ and }\ m_3=m_5 = m_4 - a = 8a.\]
We then have that
$$(A_F^t-I_2)\overrightarrow{x_a}^t
=\left(\begin{tabular}{cc}
1&0\\
1&-1\\
\end{tabular}\right)\left(\begin{tabular}{c}
6$a$\\
5$a$\\
\end{tabular}\right)=\left(\begin{tabular}{c}
6$a$\\
$a$\\
\end{tabular}\right)
\geq \left(\begin{tabular}{c}
$a$\\
$a$\\
\end{tabular}\right)$$ and
$$(A_G^t-I_3)\overrightarrow{y_a}^t=\left(\begin{tabular}{ccc}
-1&1&0\\
0&1&1\\
0&1&-1\\
\end{tabular}\right)\left(\begin{tabular}{c}
8$a$\\
9$a$\\
8$a$\\
\end{tabular}\right)=\left(\begin{tabular}{c}
$a$\\
17$a$\\
$a$\\
\end{tabular}\right)\geq \left(\begin{tabular}{c}
$a$\\
$a$\\
$a$\\
\end{tabular}\right).$$
Furthermore,
\begin{equation*}A_E^t-I_5
=\left(\begin{tabular}{ccccc}
1&0&0&0&0\\
1&-1&0&0&0\\
0&1&-1&1&0\\
0&0&0&1&1\\
0&0&0&1&-1\\
\end{tabular}\right)
=\left(\begin{tabular}{cc}
$A_F^t-I_2$&$A_{21}$\\
$A_{12}$&$A_G^t-I_3$
\end{tabular}\right)
\end{equation*}
where $A_{12}=\left(
                \begin{array}{cc}
                  0 & 1 \\
                  0 & 0 \\
                  0 & 0 \\
                \end{array}
              \right)
$ and $A_{21}= \left(
                 \begin{array}{ccc}
                   0 & 0 & 0 \\
                   0 & 0 & 0 \\
                 \end{array}
               \right)
$.

Let $\overrightarrow{m_a}:=[6a\ \
5a\ \ 8a\ \ 9a\ \ 8a]\in \mathbb{N}^5$. We then get
$$(A_E^t-I_5)\overrightarrow{m_a}^t=\left(\begin{tabular}{ccccc}
1&0&0&0&0\\
1&-1&0&0&0\\
0&1&-1&1&0\\
0&0&0&1&1\\
0&0&0&1&-1\\
\end{tabular}\right)\left(\begin{tabular}{c}
6$a$\\
5$a$\\
8$a$\\
9$a$\\
8$a$\\
\end{tabular}\right)=\left(\begin{tabular}{c}
6$a$\\
$a$\\
6$a$\\
17$a$\\
$a$\\
\end{tabular}\right)\geq \left(\begin{tabular}{c}
$a$\\
$a$\\
$a$\\
$a$\\
$a$\\
\end{tabular}\right),
$$

\noindent
which concludes the example.  \hfill $\Box$ 
\end{example}

We are now in position to give a necessary and sufficient condition for the Leavitt path algebra
of a finite source-free graph to have  Unbounded Generating Number.
We recall an important property of the monoid
$M_E$. Let $E$ be a finite graph having $|E^0| = h$, and regular (i.e., 
non-sink) vertices $\{v_i\ |\ 1\leq i\leq z\}$. For $x = n_1 v_1 + \cdots + n_h v_h \in Y_E$ (the free abelian monoid on generating set $E^0$),  and $1\leq i \leq z$, let  $M_i(x)$ denote the
element of $Y_E$ which results by applying to $x$  the relation $(M)$ (given in Definition \ref{Tdefn}) corresponding to vertex
$v_i$.   For any sequence $\sigma$ taken from $\{1, 2,..., z\}$,
and any $x\in Y$, let $\Lambda_{\sigma}(x)\in Y$ be the element which
results by applying relation $(M)$  in the order specified by $\sigma$.     

\medskip

\noindent
{\bf The Confluence Lemma.}  (\cite[Lemma 4.3]{amp:nktfga})    For each pair $x, y\in Y_E$,   $[x] = [y]$ in $M_E$ if and only if there are 
sequences $\sigma$, $\sigma'$ taken from $\{1, 2,..., z\}$ such that $\Lambda_{\sigma}(x)
= \Lambda_{\sigma'}(y)$ in $Y_E$.  

\medskip

Here is the key precursor to our main result.

\begin{thm}\label{thm4.5}
Let $E=(E^0,E^1,r,s)$ be a finite source-free graph and $K$ any field. Then
$L_K(E)$ has Unbounded Generating Number if and only if $E$
contains a source cycle.
\end{thm}
\begin{proof}

 We denote $E^0$ by $\{v_1,v_2,...,v_h\}$, in such a way that the non-sink vertices of $E$ appear as $v_1, ..., v_z$.

\smallskip

$(\Longleftarrow )$ Assume that $E$ contains a source cycle $c$;  we prove  that $L_K(E)$ has Unbounded
Generating Number. We use  Corollary \ref{cor4.2} to do so. Namely,
let $m$ and $n$ be positive integers such that
\begin{center}
$m[\sum_{i=1}^hv_i]+[x]=n[\sum_{i=1}^hv_i]$ in $M_E$
\end{center}
for some $[x]\in M_E$. We must show that $m\leq n$.  We write $x\in Y_E$
as  $$x=\sum_{i=1}^hn_iv_i,$$ where $n_i\ (i = 1,..., h)$ are nonnegative
integers. By the Confluence Lemma and and the hypothesis $m[\sum_{i=1}^hv_i]+[x]=n[\sum_{i=1}^hv_i]$, there are
two sequences $\sigma$ and $\sigma'$ taken from $\{1,...,z\}$ for which
$$\Lambda_{\sigma}(\sum_{i=1}^h(m + n_i)v_i)=\gamma=\Lambda_{\sigma'}(n\sum_{i=1}^hv_i)$$
for some $\gamma\in Y$. But each time a substitution of the form
$M_j \ (1\leq j\leq z)$ is made to an element of $Y$, the effect on that element is to:
\begin{itemize}
\item[(i)] subtract $1$ from the coefficient on $v_j$;
\item[(ii)] add $a_{ji}$ to the coefficient on $v_i$ (for $1\leq i\leq h$).
\end{itemize}
For each $1 \leq j \leq z$, denote the number of times that $M_j$ is invoked in $\Lambda_{\sigma}$ (resp., $\Lambda_{\sigma'}$) by  $k_j$ (resp., $k'_j$). Recalling
the previously observed effect of $M_j$ on an element of $Y$, we
see that
\begin{equation*}
\begin{array}{rcl}
\gamma &=& \Lambda_{\sigma}(\sum_{i=1}^h(m + n_i)v_i)  \smallskip \\  
&=& ((m+n_1-k_1)+ a_{11}k_1+a_{21}k_2+...+a_{z1}k_z)v_1\\
&&+ ((m+n_2-k_2)+ a_{12}k_1+a_{22}k_2+...+a_{z2}k_z)v_2 \\
&& + \ \  \cdots \\
&&+ ((m+n_z-k_z)+ a_{1z}k_1+a_{2z}k_2+...+a_{zz}k_z)v_z\\
&&+ ((m+n_{z+1})+  a_{1(z+1)}k_1+a_{2(z+1)}k_2+...+a_{z(z+1)}k_z)v_{z+1}  \\
&& + \ \ \cdots \\
&&+ ((m+n_{h})+a_{1h}k_1+a_{2h}k_2+...+a_{zh}k_z)v_h .
\end{array}
\end{equation*}
On the other hand, we have
\begin{equation*}
\begin{array}{rcl}
\gamma &=& \Lambda_{\sigma'}(n\sum_{i=1}^hv_i) \smallskip \\
&=& ((n-k'_1)+ a_{11}k'_1+a_{21}k'_2+...+a_{z1}k'_z)v_1\\
&& +\ ((n-k'_2)+ a_{12}k'_1+a_{22}k'_2+...+a_{z2}k'_z)v_2 \\
&& + \ \ \cdots \\
&& +\ ((n-k'_z)+ a_{1z}k'_1+a_{2z}k'_2+...+a_{zz}k'_z)v_z\\
&& +\ (n\ +\  a_{1(z+1)}k'_1\ +\ a_{2(z+1)}k'_2\ +\ ...\ +\ a_{z(z+1)}k'_z)v_{z+1} \\
&& + \ \ \cdots  \\
&& +\ (n+a_{1h}k'_1+a_{2h}k'_2+...+a_{zh}k'_z)v_h .
\end{array}
\end{equation*}   For each $1 \leq i \leq z$, define  $m_i :=k'_i-k_i$.  Then from the above observations, equating coefficients on the free generators $\{v_i \ | \ 1\leq i \leq h\}$ of $Y_E$, we get the following system of equations:
\begin{equation}\label{14}
\left\{
\begin{array}{rcl}
m-n+n_1 & = &(a_{11}-1)m_1+a_{21}m_2+...+a_{z1}m_z \\
m-n+n_2 & = & a_{12}m_1+(a_{22}-1)m_2+...+a_{z2}m_z \\
& \vdots & \\
m-n+n_z & = & a_{1z}m_1+a_{2z}m_2+...+ (a_{zz}-1)m_z\\
m-n+n_{z+1} & = & a_{1(z+1)}m_1+a_{2(z+1)}m_2+...+a_{z(z+1)}m_z \\
 & \vdots  & \\
m-n+n_h & = & a_{1h}m_1+a_{2h}m_2+...+a_{zh}m_z \\
\end{array}
\right.
\end{equation}

By hypothesis $c$ is a source cycle in $E$, i.e.,  $|r^{-1}(v)| = 1$ for all $v\in c^0$.   By renumbering vertices if necessary, we may assume without loss of generality  that
$c^0=\{v_1,...,v_p\}$.   (Note that, as each vertex in $c^0$ emits at least one edge, we have that  each of $\{v_1,...,v_p\}$ is a regular vertex.)    The condition  $|r^{-1}(v)| = 1$ then yields:  

\smallskip

-  $a_{i,i+1} = 1$ for $1\leq i \leq p-1$; 
 
 - $a_{p,1} = 1$;  
 
 - $a_{j,i+1} = 0$ for $1\leq i \leq p-1$ and $j \neq i$ ($1\leq j \leq h$);  and 
 
-   $a_{j,1} = 0$ if $j\neq p$ ($1\leq j \leq h$).         

\smallskip

\noindent
      If $p =1$ (i.e., if $c$ is a loop), then $a_{11} = 1$, and  first equation in the system of equations
(1) becomes  
$$m-n + n_1 =  (1 - 1)m_1 + 0m_2 + \cdots + 0m_z = 0,$$ so $m-n = -n_1\leq 0$, i.e., $m\leq n$.

If $p \geq 2$, then using the noted   information about the $a_{i,j}$,  the $p$ first equations of the system
of equations (1) can be written as:
\begin{equation*}\label{15}
\left\{
\begin{array}{rclrrrrr}
m-n+n_1 & = &-m_1 & & & & &+m_p \\
m-n+n_2 & = & m_1 &-m_2 & & & & \\
m-n+n_3 & = &      & m_2&-m_3 & & & \\
 & \vdots & \\
m-n+n_p & = &      & & & &m_{p-1} & -m_p\\
\end{array}
\text{.}
\right.
\end{equation*}
Then  adding both sides yields that $p(m-n)  + (n_1+...+n_p) = 0$, so that  $p(m-n) = -(n_1+...+n_p) \leq 0$, which gives $m\leq n$.

Therefore, $L_K(E)$ has Unbounded Generating Number.

\medskip

$(\Longrightarrow)$  Assume conversely that $E$ does not contain any source cycles.   We will prove that $L_K(E)$
does not have Unbounded Generating Number.  So let $m$ and $n$ be two positive
integers such that $m> n$. We will establish the existence of  an element $x=\sum_{i=1}^hn_iv_i\in Y_E$
such that $$m[\sum_{i=1}^hv_i]+[x]=n[\sum_{i=1}^hv_i]$$
in $M_E$. 
Equivalently, arguing as in the previous half of the proof,  we show that we can find nonnegative integers $n_i,\ k_j$ and $k'_j$ ($i=1, ..., h$
and $j= 1,..., z$) such that
\begin{equation}\label{14}
\left\{
\begin{array}{rcl}
m-n+n_1 & = &(a_{11}-1)m_1+a_{21}m_2+...+a_{z1}m_z \\
m-n+n_2 & = & a_{12}m_1+(a_{22}-1)m_2+...+a_{z2}m_z \\
&  \vdots & \\
m-n+n_z & = & a_{1z}m_1+a_{2z}m_2+...+ (a_{zz}-1)m_z\\
m-n+n_{z+1} & = & a_{1(z+1)}m_1+a_{2(z+1)}m_2+...+a_{z(z+1)}m_z \\
 & \vdots  & \\
m-n+n_h & = & a_{1h}m_1+a_{2h}m_2+...+a_{zh}m_z \\
\end{array}
\right.
\end{equation}
where $m_j :=k'_j-k_j$ for all $j = 1, ..., z$.

We apply Lemma \ref{lem 4.3} to find such elements. Namely,
let $F$
%=(F^0, F^1,\\ r|_{F^1}, s|_{F^1})$ 
be the subgraph of $E$ defined by:
$$F^0 :=\{v_1,..., v_z\} \text{ and }F^1:=r^{-1}(F^0).$$
 In other words, $F$ is the
graph produced from $E$ by removing the sinks.  Specifically, we have that $A^t_F$ is the $z \times z$ matrix 
$$A^t_F= \left(\begin{array}{ccc}
    a_{11} & ... & a_{z1}  \\
    a_{12} & ... & a_{z2}\\
    \vdots & \cdots & \vdots  \\
    a_{1z} & ... & a_{zz}
  \end{array}
\right)
$$   Also, we note that the first $z$  equations of the system of equations (\ref{14})  in the proof
of Theorem \ref{thm4.5}  is induced by the matrix $A^t_F - I_z$. 
 Easily we see that 
 $F$ contains neither sources nor  source cycles (because $E$ contains neither).     By Lemma \ref{lem 4.3} (applied to the graph $F$ and positive integer $a = m-n$), there is a row vector $\overrightarrow{m} =
[m_1\ ...\ m_z]\in \mathbb{N}^z$ such that $m_j\geq m-n$ for all $j=1,...,z$, and

$$(A_F^t-I_z)\overrightarrow{m}^t
\geq [m-n\ ...\ m-n]^t.$$ That is, we have
$$a_{1j}m_1+ ... +(a_{jj}-1)m_j+...+a_{zj}m_z\geq m-n$$ for all $j=1,...,z$.
For each $j=1, ..., z$, let
$$n_j:= a_{1j}m_1+ ... +(a_{jj}-1)m_j+...+a_{hj}m_h - (m - n).$$
For each $j=z+1,...,h$, as $v_j$ is not a source, there exists $i\in \{1,...,z\}$
such that $a_{ij}\geq 1$, and hence, for such $j$, 
$$a_{1j}m_1+a_{2j}m_2+...+a_{zj}m_z\geq a_{ij}m_i\geq m_i\geq m-n.$$
We then  choose the non-negative integers  $n_j$ $(j=z+1,...,h)$ as follows:
$$n_j:= a_{1j}m_1+a_{2j}m_2+...+a_{zj}m_z -(m-n).$$

Finally, positive integers $k_j$ and $k'_j$ $(j = 1, ..., z)$ are chosen arbitrarily
such that $m_j = k'_j - k_j$ for all $j=1,..,z$. Then,  a tedious but straightforward computation yields that this choice of integers indeed satisfies the system of equations (2) above, thus completing the proof of the theorem.      
\end{proof}

\begin{rem}
We have in fact shown in the proof of Theorem \ref{thm4.5}  that the UGN property fails for the order-unit $[\sum_{i=1}^hv_i]$ of $M_E$ for {\it every} pair of positive integers $m>n$; of course, it was required only to  show that it fails for {\it some} such pair.   Using the previously-mentioned {\it separativity} of $\mathcal{V}(L_K(E))$ (and so of $M_E$), one can easily show that failure of UGN for one pair is equivalent to failure of UGN for every pair.   \hfill $\Box$ 
\end{rem}

\begin{example}\label{NOTUGNExample}
We present a specific example of the construction presented in the proof of Theorem \ref{thm4.5} which shows that source-free graphs having no source cycles do not have UGN.  Let $K$ be a field and let $E $ be the graph

$$ 
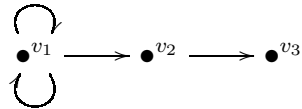
\xymatrix{   \bullet^{v_1} \ar[r]     \ar@(ul,ur) \ar@(dr,dl)           & \bullet^{v_2} \ar[r]  & \bullet^{v_3} } $$

\medskip

%as follows:
%\[E^0 = \{v_1, v_2, v_3\} \text{ and } E^1 = \{f_1, f_2, f_3, f_4\},\] where
%$s(f_i) =  v_1 (1\leq i\leq 3), s(f_4) = v_2$ and $r(f_1) = v_1= r(f_2), r(f_3) = v_2, r(f_4) = v_3$.
\noindent
Clearly, $E$ is a source-free graph in which no cycle is a source cycle,  and
$$A^t_E=\left(
    \begin{array}{ccc}
      2 & 0 & 0 \\
      1 & 0 & 0 \\
      0 & 1 & 0 \\
    \end{array}
  \right).
$$ We will show that $L_K(E)$ does not have UGN. So let $m$ and $n$ be two positive integers such that
$m > n$. We will establish the existence of an element $x = \sum_{i=1}^3n_iv_i\in Y_E$ such that
\[m[\sum_{i=1}^3v_i] + [x] = n[\sum_{i=1}^3v_i].\] Equivalently, we show that
we can find nonnegative integers $n_i, k_j$ and $k'_j$ $(1\leq i\leq 3$ and $1\leq j\leq 2)$ such that
\begin{equation}
\left\{
\begin{array}{rcl}
m-n+n_1 & = &m_1\\
m-n+n_2 & = & m_1 - m_2\\
m-n+n_3 & = & m_2\\
\end{array}
\right.
\end{equation}
where $m_j :=k'_j-k_j$ for  $j = 1, 2$. Let $F$ be the graph produced  from $E$ by deleting $v_3$.
Note that \begin{equation*}A^t_F
=\left(\begin{tabular}{cccc}
2 & 0\\
1 & 0 \\

\end{tabular}\right)
\end{equation*} and

\begin{equation*}A^t_E
=\left(\begin{tabular}{cccc}
2 & 0 & 0 \\
1 & 0 & 0 \\
0 & 1 & 0 \\
\end{tabular}\right)
=\left(\begin{tabular}{ccc}
$A^t_F$ & \vdots &0\\
 ... &  &0\\
0&1&0\\
\end{tabular}\right).
\end{equation*} Also, the two first equations of the above system can be written as
$$(A^t_F-I_2)\left(\begin{tabular}{cc}
$m_1$\\
$m_2$\\
\end{tabular}\right)= \left(\begin{tabular}{cc}
$m - n + n_1$\\
$m - n + n_2$\\
\end{tabular}\right).$$

\noindent
As in the proof of Theorem \ref{thm4.5}, we define $m_1$ and $m_2$ as follows: \[m_1 = 3|F^0|(m-n) = 6(m-n), \ \text{ and } \ 
m_2 = m_1 - (m-n) = 5(m-n).\]    Subsequently, we define 
\[n_1 = m_1 - (m-n) = 5(m-n),\] \[n_2 = m_1-m_2- (m-n) =0, \text{ and }\]
\[n_3 = m_2 - (m-n) = 4(m-n).\] 

\noindent
Then the construction described in the proof of Theorem \ref{thm4.5} yields that the element  $$[x] = [5(m-n)v_1 + 4(m-n)v_3]$$ of $M_E$ satisfies  $$m[\sum_{i=1}^3v_i] + [x] = n[\sum_{i=1}^3v_i]$$ in $M_E$.    
It is instructive to verify the validity of this equation directly;  we achieve this by verifying the equivalent version   $$ [(6m - 5n)v_1 + mv_2 + (5m - 4n)v_3] = [nv_1 + nv_2 + nv_3]$$ in $M_E$.  In $M_E$ we have:  
$$ \mbox{(i)} \ \   [v_1] = [2v_1 + v_2],  \ \ \ \  \mbox{and} \ \ \ \  \mbox{ (ii)}  \  \ [v_2] = [v_3].$$   Recall  that $m_1 = 6(m - n)$ and $m_2 = 5(m - n)$. We must choose positive integers $k_i$ and $k'_i$  such that $m_i = k'_i - k_i$  $(i = 1, 2)$.  We choose these as follows: $k_1 = 1 = k_2$, $k'_1 = 6(m - n) +1$, and $k'_2 = 5(m - n) +1$. 
The left side can be transformed as follows:

\medskip

\qquad $[(6m - 5n)v_1 + mv_2 + (5m - 4n)v_3]$  
\begin{eqnarray*}
& = &  [(6m - 5n - 1)v_1 + v_1 + mv_2 + (5m - 4n)v_3]  \\
& = &  [(6m - 5n - 1)v_1 + (2v_1  + v_2)+ mv_2 + (5m - 4n)v_3]   \qquad  \ \ \mbox{by (i)}\\
& = &  [(6m - 5n +1)v_1   + mv_2 + v_2 + (5m - 4n)v_3]  \\
& = &  [(6m - 5n +1)v_1   + mv_2 + v_3 + (5m - 4n)v_3]   \qquad \qquad \qquad   \mbox{by (ii)}\\
& = &  [(6m - 5n +1)v_1   +mv_2 + (5m - 4n + 1)v_3].\\
\end{eqnarray*}

\vspace{-.2in}

On the other hand, an application of (i) yields $$[nv_1 + nv_2] = [(n-1)v_1 + v_1 + nv_2] = [(n-1)v_1 + (2v_1 + v_2) + nv_2] = [(n+1)v_1 + (n+1)v_2] .$$   By an easy induction this gives $$[nv_1 + nv_2] = [(n+ u)v_1 + (n+u)v_2]$$ for every $u\in \mathbb{N}$; in particular, applying (i)  $u= k_1' = 6(m-n)+1$ times gives the first step in the following transformation of the right side:

\medskip

\qquad $  [nv_1 + nv_2 + nv_3] $ 
\begin{eqnarray*}
\qquad & = &  [(6m - 5n + 1)v_1 + (6m - 5n + 1) v_2 + nv_3]  \\
& = &  [(6m - 5n - 1)v_1 + mv_2 + (5m - 5n + 1) v_2   + nv_3]   \\
& = &  [(6m - 5n - 1)v_1 + mv_2 + (5m - 5n + 1) v_3   + nv_3] \qquad     \mbox{by (ii), } k_2' \mbox{ times}  \\
& = &  [(6m - 5n +1)v_1   +mv_2 + (5m - 4n + 1)v_3]. \\
\end{eqnarray*}

\vspace{-.2in}

\noindent 
This completes the verification that the two quantities are indeed equal in $M_E$.    
\hfill $\Box$

\end{example}
In our main result (Theorem \ref{thm4.10}), we show how to eliminate the ``no sources" hypothesis in Theorem \ref{thm4.5}.   

\begin{defn}[{e.g.,  \cite[Notation 2.4]{ar:fpsmolpa}}] 
%\cite[Definition 1.2]{aps:fiitcofpa} and
Let $E = (E^0, E^1, r, s)$ be a graph,
and let $v\in E^0$ be a source. We form the \emph{source elimination}
graph $E_{\setminus v}$ of $E$ as follows: 
$$(E_{\setminus v})^0 = E^0\setminus \{v\}; \ (E_{\setminus v})^1= E^1\setminus s^{-1}(v); 
\ s_{E_{\setminus v}} = s|_{(E_{\setminus v})^1}; \ \mbox{and} \ r_{E_{\setminus v}} = r|_{(E_{\setminus v})^1}.$$ 
 In other words, $E_{\setminus v}$ denotes
the graph gotten from $E$ by deleting $v$ and all of edges in $E$ emitting from $v$.    \hfill $\Box$ 
\end{defn}

Let $E$ be a finite graph. If $E$ is acyclic, then repeated application of the source
elimination process to $E$ yields the empty graph. On the other hand, if $E$ contains a cycle,
then repeated application of the source elimination process will yield a source-free graph
$E_{sf}$ which necessarily contains a cycle.

Consider the sequence of graphs which arises in some step-by-step process of source
eliminations
\[E := E_0\rightarrow E_1\rightarrow\cdots\rightarrow E_i\rightarrow\cdots
\rightarrow E_\ell : = E_{sf} .\] To avoid defining a graph to be the empty set, we define
$E_{sf}$ to be the graph $E_{triv}$ (consisting of one vertex and no edges) in case
$E_{\ell - 1} = E_{triv}$. 

%The following proposition shows that the graph $E_{sf}$ is unique.

%%%  Moved the old proof past the \end{document} line ..

Although there in general are many different orders in which a step-by-step source elimination process can be carried out,  the resulting source-free subgraph $E_{sf}$ is always the same.     For a cycle $c$ in $E$, denote by $T_E(c)$ the set of vertices 
$$   T_E(c) := \{w \in E^0 \ | \ v \geq w \mbox{ for some }v \in c^0\}.$$  

\begin{lem}
Let $E$ be a finite graph.   

(1)   $E_{sf} = E_{triv}$ if and only if $E$ is acyclic.

(2)   Suppose $E$ contains cycles.   Then 
$$E_{sf}^0 = \bigcup_{c} T_E(c) \ \ \mbox{(where } c \mbox{ runs over all cycles in }E),$$
and $E_{sf}^1 = E^1|_{E_{sf}^0}.$  
\end{lem}

\begin{proof}   We first note that, if $v\in E^0$ is a source, then it is easy to verify that 

\qquad (1) $c$ is a cycle in $E$ if and only if $c$ is a cycle in $E_{\setminus v}$, and 

\qquad (2) if $c$ is a cycle in $E$, then $T_E(c) = T_{E_{\setminus v}}(c).$

Now consider the sequence of graphs which arises in some step-by-step process of source
eliminations
\[E:= E_0\rightarrow E_1\rightarrow\cdots\rightarrow E_i\rightarrow\cdots
\rightarrow E_\ell =: E_{sf}.\]
Using the two observations, we immediately get that $E_{sf} = E_{triv}$ if and only if $E$ is
acyclic; and that  $T_E(c)\subseteq E^0_{sf}$ for each cycle $c$ in $E$. Now assume that
$v\notin E^0_{sf}$.  Let $i$ be minimal in $\{1,..., \ell\}$ such that 
$v\notin E_i$, and hence, $v$ is a source in $E_{i-1}$. This implies that $v\notin T_{E_{i-1}}(c)$
for each cycle $c$ in $E_{i-1}$. Applying the second observation, we get that $v\notin T_{E}(c)$
for each  cycle $c$ in $E$. 
%From these observations, we immediately get that  $E_{sf} = \bigcup_cE(c)$, where $c$ runs over all cycles in $E$. 
\end{proof}

The key result which will allow the extension of Theorem~\ref{thm4.5} is the following observation 
of Ara and Rangaswamy.

\begin{lem}[{\cite[Lemma 4.3]{ar:fpsmolpa}}]\label{lem4.8}
Let $E$ be a finite graph and $K$ any field. If $v$ is a source which is not isolated, then $L_K(E)$ is Morita
equivalent to $L_K(E_{\setminus v})$.
\end{lem}

\begin{lem}\label{lem4.9}
Let $E$ be a finite graph containing an isolated vertex, and $K$ any  field. Then $L_K(E)$ has
Unbounded Generating Number.
\end{lem}
\begin{proof} Let  $v$ denote the presumed isolated vertex. We then get immediately that
$L_K(E)\cong K \oplus L_K(E_{\setminus v})$, and hence there is a natural surjection from
$L_K(E)$ onto $K$. Obviously, the field $K$ has UGN, so $L_K(E)$ has UGN by Lemma~\ref{lem 3.3}.
\end{proof}

Using Theorem~\ref{thm4.5} and Lemmas~\ref{lem4.8} and~\ref{lem4.9}, we are finally in position to establish the main result of this article.  

\begin{thm}\label{thm4.10}
Let $E$ be a finite graph and $K$ any field. Let \[E= E_0\rightarrow E_1\rightarrow\cdots\rightarrow E_i\rightarrow\cdots
\rightarrow E_\ell = E_{sf}\] be a sequence of graphs which arises in some step-by-step process of source
eliminations. Then $L_K(E)$ has Unbounded Generating Number if and only if either $E_i$ contains an isolated
vertex (for some $0\leq i\leq \ell$), or $E_{sf}$ contains a source cycle.
\end{thm}
\begin{proof} Assume first that $E_i$ contains an isolated vertex for some $i$. Let $j$ denote
the minimal such $i$. Then at each step of the source elimination process
\[E= E_0\rightarrow E_1\rightarrow\cdots\rightarrow E_j \] the source which is being eliminated
is not an isolated vertex. By Lemma~\ref{lem4.8} we then have that the algebras
$L_K(E)$, $L_K(E_1)$, ..., $L_K(E_j)$ are Morita equivalent one to the other. But $L_K(E_j)$
has UGN by Lemma~\ref{lem4.9}, and hence $L_K(E)$ has UGN by Theorem~\ref{thm 3.8}.

On the other hand, suppose that no $E_i$ contains an isolated vertex. Then Lemma~\ref{lem4.8}
applies at each step of the source elimination process, so that  $L_K(E)$ is
Morita equivalent to $L_K(E_{sf})$. So, by Theorem~\ref{thm 3.8}, $L_K(E)$ has UGN if and only
if $L_K(E_{sf})$ has UGN. As $E_{sf}$ is source-free we may apply Theorem~\ref{thm4.5}, so that
$L_K(E)$ has UGN if and only if $E_{sf}$ contains a source cycle, thus establishing the result.
\end{proof}

We emphasize that  the statement of Theorem \ref{thm4.10} depends not only on the subgraph $E_{sf}$, but on the sequence of source-eliminations as well.   As an easy example, consider the two graphs  $F$ and $G$:
$$ F =   \   \xymatrix{ \bullet  \hspace{-.15in} &   \bullet  
%\ar@(ul,ur) \ar@(dr,dl) 
 \ar@(ur,dr)  \ar@(ul,dl)   }   \ \ \ \ \ \    \ \ \mbox{and}  \ \ \ \ \ G =  \ \ \ \ \xymatrix{  \bullet   \ar@(ur,dr)  \ar@(ul,dl)   } \ \ \ \  \ .  $$ 

\noindent
(So $F$ is the disjoint union of $E_{triv}$  with $G$.)  
Then obviously $F_{sf} = G_{sf} = G$.   But $L_K(F)$ has UGN (it has a direct summand isomorphic to $K$), while $L_K(G)$ does not (as $G$ is a source-free graph containing no source cycles.)

We finish this section with a few remarks about Cohn path algebras.   
%As mentioned in the Introduction, every Cohn path algebra is, in fact, isomorphic to a Leavitt path algebra.
We present here a specific case of a more general result described in \cite[Section 1.5]{aam:lpa}. Namely,
let $E = (E^0, E^1, s, r)$ be an arbitrary graph and $\Phi$
the set of regular vertices of $E$. Let $\Phi '=\{v'\ |\ v\in \Phi \}$ be a
disjoint copy of $\Phi$. For $v\in \Phi$ and for each edge $e$ in $E^1$ such
that $r_E(e)=v$, we consider a new symbol $e'$. We define the graph
$F(E)$, as follows: $$F(E)^0 := E^0\sqcup \Phi' \text{ and } F(E)^1 :=E^1\sqcup \{e'\ |\ r_E(e)\in
\Phi \},$$ and for each $e\in E^1$, $s_{F(E)}(e)=s_E(e),\ s_{F(E)}(e')=s_E(e),\
r_{F(E)}(e)=r_E(e)$, and $r_{F(E)}(e')=r_E(e)'$.   For instance, if 
$$E =  \ \ \ \ \xymatrix{  \bullet^v   \ar@(ul,ur)^e \ar@(dr,dl)^f } \  ,\ \ \ \mbox{then}  \ \ \ F(E) = \xymatrix{  \bullet^v   \ar@(ul,ur)^e \ar@(dr,dl)^f \ar@/^.5pc/[r]^{e'}      \ar@/^-.5pc/[r] _{f'}        &\bullet^{v'}   } .$$
\noindent
Ara, Siles Molina and the first author have shown  that for any graph $E$ and any field $K$,  there is an
isomorphism of $K$-algebras 

\medskip

\hspace{1.5in} $C_K(E)\cong L_K(F(E)).$  \hfill  (cf. \cite[Theorem 1.5.17]{aam:lpa})

\medskip
\noindent
(So, perhaps counterintuitively, every Cohn path algebra is in fact isomorphic to a Leavitt path algebra.)  From this observation
and Theorem~\ref{thm4.10}, we immediately get a criterion to determine which 
 Cohn path algebras of finite graphs have Unbounded Generating Number.

\begin{cor}\label{cor4.11}
Let $E$ be a finite graph and $K$ any field. Then
the Cohn path algebra $C_K(E)$ has Unbounded Generating Number if and only if $E$
satisfies at least one of the following conditions:
\begin{enumerate}
\item[(1)] $E$ contains a source; or 

\item[(2)] $E$ contains a source cycle. 
\end{enumerate}
\end{cor}
\begin{proof} ($\Longrightarrow$)  Assume that $E$  contains neither sources nor source cycles.   By the construction
of the graph $F(E)$, it is easy to see that $F(E)$ also  contains neither  sources nor source cycles.  
Then, by Theorem \ref{thm4.10}, $C_K(E)\cong L_K(F(E))$ does not have UGN. 

($\Longleftarrow$) If $E$ contains a source cycle $c$,  then by considering the explicit construction given above, it is clear that  the graph $F(E)$
must contain such a cycle, and hence $F(E)_{sf}$ contains such a cycle too.
So  by Theorem \ref{thm4.10}, $L_K(F(E))$ has UGN, and hence  so does $C_K(E)$.

On the other hand, suppose $E$ contains a source vertex $v$. If
$v$ is a sink (i.e., if $v$ is isolated) in $E$ then $v$ is an isolated vertex in $F(E)$. Otherwise, since $v$ is a
source in $E$, the corresponding  vertex $v'$ is an isolated vertex in $F(E)$. Hence, in this case,
$F(E)$ always contains an isolated vertex. Then, by 
%Theorem \ref{thm4.10}, 
Lemma \ref{lem4.9}, $L_K(F(E))$
has UGN, hence so does $C_K(E)$. 
\end{proof}

In \cite[Theorem 9]{ak:cpahibn}, Kanuni and the first author showed that
the Cohn path algebra of any finite graph has the IBN property. Using this result and Theorem \ref{thm4.10},
we easily give an example of a Leavitt path algebra which has the IBN property, but does not
have the UGN property.

\begin{example}\label{IBNnotUGN}   Let $K$ be a field, and let $G$ be the graph 

$$
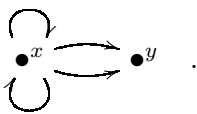
\xymatrix{  \bullet^x   \ar@(ul,ur) \ar@(dr,dl) \ar@/^.5pc/[r]      \ar@/^-.5pc/[r]         &\bullet^y  } \   \ .$$

\bigskip

%as follows:
%\[E^0 = \{v, w\}\ \text{ and }\ E^1 = \{f_1, f_2, f_3, f_4\},\] where $s(f_i) = v\ (1\leq i\leq 4)$,
%$r(f_1) = v = r(f_2)$ and $r(f_3) = w = r(f_4)$. 
\noindent
Then, by Theorem \ref{thm4.10},
$L_K(G)$ does not have UGN.   But as shown previously, $G = F(E)$ where $E$ is the graph 
$$
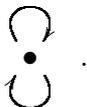
\xymatrix{  \bullet   \ar@(ul,ur) \ar@(dr,dl)   } \ \ \ .$$

\bigskip
\noindent
So $L_K(G) \cong C_K(E)$, and so has IBN by \cite[Theorem 9]{ak:cpahibn}.   

It is perhaps instructive to explicitly consider $\mathcal{V}(L_K(G)) \cong M_G$ in this case.   Specifically, $M_G $ is the free abelian monoid on $\{x,y\}$ with the one relation $x = 2x + 2y$.    The standard order-unit  of $M_G$ is $[x+y]$.   It is not hard to show that any equation of the form $n[x+y] = m[x+y]$ in $M_G$ necessarily gives $m=n$. (The one relation, applied to an element of $Y_G$ of the form $t = nx + ny$, will either yield $t$ itself, or an element $t' = ix + jy$ for which $i\neq j$.)   On the other hand, the relation $x = 2x + 2y$ gives $[x + y] = [2x + 3y] = 2[x+y] + [y]$ in $M_G$, so that $[x+y]$ does not have the UGN property in $M_G$.  \hfill $\Box$ 
%On the other hand, consider the subgraph $G$ of $E$ as follows: \[G^0 = \{v\}\ \text{ and }\
%G^1 = \{f_1, f_2\}.\] Then, we have that $F(G)$ is isomorphic to $E$ as graphs, and hence,
%$L_K(E)\cong L_K(F(G)) \cong C_K(G)$. This implies that $L_K(E)$ has IBN by \cite[Theorem 9]{ak:cpahibn}.
\end{example}

%%%%%%%%%%%%%%%%%%%%%
%%%%%%%%%%%%%%%%%%%%%
%
%  Four equivalent conditions section
%
%%%%%%%%%%%%%%%%%%%%%
%%%%%%%%%%%%%%%%%%%%

\section{Leavitt path algebras having cancellation of projectives}\label{4moreconditionssection}
In this, the article's short final section, we identify the graphs $E$ for which the  Leavitt path algebra $L_K(E)$ satisfies conditions (3) through (5) mentioned in the Introduction.    We briefly review some terminology.

\smallskip

Let $E$ be a graph, and $p  = e_{1} \cdots e_{n}$ a path in $E$.  Then an edge $f\in E^1$ is an  \emph{exit} for $p$ 
 if $s(f) = s(e_{i})$ but $f \ne e_{i}$ for
some $1 \le i \le n$. 
 $E$ is said to be a \emph{no-exit graph} if no cycle in $E$ has an exit.

% More precisely, we completely describe Leavitt path algebras of finite graphs have any one of stably finite, Hermite and cancellation of projectives properties.

A ring $R$ is called \emph{directly finite} if, for any $a, b\in R$, $ab=1$ implies
$ba=1$.

%We begin by recalling some fundamental notions. 

 $R$ is said to be \emph{stably finite}  if for any $n\in \mathbb{N}^+$, $R^n
\cong R^n\oplus K$ (as right $R$-modules) implies $K=0$.
% (see, e.g., \cite[Proposition 1.7]{l:lomar}).  

%Notice that every commutative ring is always stably finite (see, e.g., \cite[Proposition 1.12]{l:lomar}), and every stably finite ring has Unbounded Generating Number (see, e.g., \cite[Proposition 1.22]{l:lomar}).

 $R$ is called a \emph{Hermite ring}
% (in P. M. Cohn's sense)
 if for all $m, n \in \mathbb{N}^+$ 
and any right $R$-module $K$, $R^n\cong R^m\oplus K$ (as right $R$-modules) implies that $n\geq m$ and $K\cong R^{n-m}$.

Finally, $R$ is said to have \emph{cancellation of projectives} if  for any finitely generated
projective right $R$-modules $P$ and $P'$, $P\oplus R\cong P'\oplus R$ (as right $R$-modules) implies that
$P\cong P'$.

\smallskip

Short, straightforward computations immediately establish that, for any unital ring $R$, 
\begin{center}
cancellation of projectives $\Rightarrow$ Hermite $\Rightarrow$ stably finite $\Rightarrow$ directly finite.
\end{center}
\noindent 
For general rings, there are examples which show that none of these implications can be reversed.   
Germane here is  the observation  that it is easy to establish that 
\begin{center}
 stably finite $\Rightarrow$ Unbounded Generating Number;
\end{center}
\noindent
however, examples exist which show that directly finite does not in general imply Unbounded Generating Number  (nor does UGN imply directly finite).    

An abelian monoid is called {\it cancellative} in case, for every $m, m', m'' \in M$, if $m' + m = m'' + m$, then $m' = m''$.  Obviously the monoid $\mathbb{N}$ is cancellative; almost as obviously, so too is $\mathbb{N}^t$ for any positive integer $t$.    
%Clearly  the cancellation of projectives property on a ring $R$  is equivalent to the cancellative property on the monoid  $\mathcal{V}(R)$.  

%Every ring having cancellation of projectives is obviously Hermite; and every Hermite ring is stably finite; but the converses are not necessarily true. Also, it is not at all easy to decide whether a given ring has any one of these properties. Below we will give a description of Leavitt path algebras of finite graphs having any one of these properties, which shows that these properties are the same for those Leavitt path algebras. Before doing so, we note the following obvious fact.

\begin{rem}\label{rem5.1}
A ring $R$ has cancellation of projectives if and only if the monoid $\mathcal{V}(R)$ is cancellative.   This is not hard to see.  Indeed, assume that $R$ has cancellation of projectives as defined above, and suppose 
$[P] + [Q] = [P'] + [Q]$ in $\mathcal{V}(R)$, i.e.,  $P\oplus Q\cong
P'\oplus Q$ as right $R$-modules. Since $[R]$ is an order-unit in  $\mathcal{V}(R)$, 
there exist $n \in \mathbb{N}^+$  and a right $R$-module $K$ such that $Q\oplus K
\cong R^n$. But then
$P\oplus R^n\cong P\oplus Q\oplus K\cong P'\oplus Q\oplus K\cong P'\oplus R^n,$
so $P\cong P'$ as $R$ has cancellation of projectives, i.e., $[P]=[P']$.    The other implication is immediate.     \hfill $\Box$ 
\end{rem}  

Here is the relationship between these properties in the context of Leavitt path algebras.  

\begin{thm}\label{thm5.2}
Let $E$ be a finite graph and $K$ any field. Then the following are equivalent:
\begin{enumerate}
\item[(1)] $L_K(E)$ has cancellation of projectives;
\item[(2)] $L_K(E)$ is Hermite;
\item[(3)] $L_K(E)$ is stably finite;
\item[(4)] $L_K(E)$ is directly finite;
\item[(5)]  $E$ is a no-exit graph.
%\item[(6)] $L_K(E)$ is Noetherian.
\end{enumerate}
\end{thm}
\begin{proof}  
% The equivalence of statements   (5)  and (6) is established in \cite{apm:lflpa}.  
 By the discussion above, 
we need only  show that (4) implies (5), and (5) implies (1).   
That (4) and (5) are equivalent is established in \cite[Theorem 4.12]{Vas}.   We give here a very brief outline of one direction.   Suppose  $E$ contains a cycle $c$ with an exit $f$, and let $v = s(f)$.  We may view $c$ as being based at $v$.  Then  $c^*c = v$.   Let $x:= \sum_{w\in E^0, w\neq v}w \in L_K(E).$  Let $a:= c + x$
and $b:= c^* +x $. It is easily verified that
$ba = 1$ in $L_K(E)$.   But $ab = cc^* +x$; since $c^*f = 0$ we get $abf = 0$, so in particular  $ab \neq 1$, so 
% $f$ can then easily be used to show that  that $ab \neq 1$, so
 $L_K(E)$ is not directly finite. 

%(4)$\Longrightarrow$(5). Assume that (4) holds and $E$ contains a cycle
%$p$ with an exit $f$. Let $v$ be the base of the cycle.
%By \cite[Lemma 2.2]{ap:tlpaoag05}, we then have
%that $p^*p =  v$. Notice that $1 = \sum_{w\in E^0}w$ is the unit
%for $L_K(E)$ (see \cite[p. 474]{tomf:lpawciacr}), $wp = 0 = pw$,
%$wp^* = 0 = p^*w$ for all $w\in E^0\setminus\{v\}$, $vp = p = pv$
%and $vp^* = p^* = p^*v$.

%Let $a:= p + \sum_{w\in E^0,\ w\neq v}w$
%and $b:= p^* +\sum_{w\in E^0,\ w\neq v}w$. It is easily verified that
%$ba = 1$. By our hypothesis, $ab = 1$, and hence, we have
%$$pp^* + \sum_{w\in E^0,\ w\neq v}w = 1= \sum_{w\in E^0}w.$$
%This implies that $pp^* = v$. Let $p := e_1...e_n$. Then, by $f$ is an exit of $p$,
%there exists $i\in \{1,..., n\}$ such that $s(e_i) = s(f)$ and $e_i\neq f$. Set $z:= e_1...
%e_{i-1}f$. We have that $s(z) = v$ and $p^*z = 0$, and hence,
%$z = vz = pp^*z = 0$. But then $$r(z) = z^*z = 0,$$ contradicting that
%$r(z)\neq 0$, by \cite[Proposition 3.4]{tomf:lpawciacr}.

So now  suppose that (5) holds; we show that $L_K(E)$ has cancellation of projectives, i.e., we show that $\mathcal{V}(L_K(E))$ is a cancellative monoid. Let $\{c_1,..., c_\ell\}$ and
$\{v_1, ..., v_k\}$  be the sets of cycles and sinks in $E$, respectively. (Because $E$ is a no-exit graph, the cycles in $E$ are necessarily disjoint.)  Then, by
\cite[Theorems 3.8 and 3.10]{apm:lflpa} 
%(see, also, \cite[Theorem 2.1]{al:nlpaatra}), 
we get 
$$L_K(E) \  \cong \  (\bigoplus_{i=1}^\ell M_{m_i}(K[x,x^{-1}]))\oplus(\bigoplus_{j=1}^{k}M_{n_j}(K)),$$
where for each $1\leq i\leq \ell$, $m_i$ is the number of paths ending in a fixed
(although arbitrary) vertex of the cycle $c_i$ which do not contain the cycle
itself, and for each $1\leq j\leq k$, $n_j$ is the number of paths ending in
the sink $v_j$.

Every finitely generated projective $K[x, x^{- 1}]$-module $P$ is free
(see, e.g., \cite[Corollary 4.10, page 189]{l:spopm}),  and $K[x, x^{- 1}]$ has IBN, so  we immediately get that $\mathcal{V}(K[x,x^{-1}])\cong \mathbb{N}$. But then, as $\mathcal{V}$ is a Morita invariant and preserves ring direct sums, the displayed ring isomorphism yields that $\mathcal{V}(L_K(E))$ is isomorphic to the cancellative monoid  $\mathbb{N}^\ell \oplus \mathbb{N}^{k}\cong\mathbb{N}^{\ell+k}$. 
%\begin{equation*}
%\begin{array}{rl}
%\mathcal{V}(L_K(E))&\cong\mathcal{V}\left((\bigoplus_{i=1}^lM_{m_i}(K[x,x^{-1}]))
%\oplus(\bigoplus_{j=1}^{k}M_{n_j}(K))\right)\\
%&\cong(\bigoplus_{i=1}^l\mathcal{V}(M_{m_i}(K[x,x^{-1}])))
%\oplus(\bigoplus_{j=1}^{k}\mathcal{V}(M_{n_j}(K)))\\
%&\cong(\mathcal{V}(K[x,x^{-1}]))^l\oplus(\mathcal{V}(K))^k\\
%&\cong\mathbb{N}^l \oplus \mathbb{N}^{k}\cong\mathbb{N}^{l+k}.\\
%\end{array}
%\end{equation*}
%This implies that $L_K(E)$ has cancellation of projectives by Remark~\ref{rem5.1},
%which ends the proof.
\end{proof}

By \cite[Theorem 3.10]{apm:lflpa}, we may add the statement  ``$L_K(E)$ is Noetherian" to Theorem \ref{thm5.2}.   Although in general the Noetherian condition on a ring $R$ is enough to yield that $R$ is stably finite, it is not sufficient in general to yield that $R$ is Hermite (neither, then, that $R$ has cancellation of projectives).   In particular, we must utilize the explicit structure of Noetherian Leavitt path algebras (as presented in the displayed isomorphism in the proof of Theorem \ref{thm5.2}) in order to conclude the cancellation of projectives property.   

%Using Theorems~\ref{thm4.10} and \ref{thm5.2}, we easily give an example of a Leavitt path algebra which has UGN, but does not have any one of the above properties.

\begin{example}\label{Toeplitz}
Let $K$ be a field, and consider the Toeplitz graph 
$$ \mathcal{T} = \ \ \  \ \  \xymatrix{  \bullet   \ar@(dl,ul)  \ar[r]        &\bullet  }$$

\noindent described in Examples \ref{monoidexamples}.  By Theorem \ref{thm4.10}, $L_K(E)$ has UGN.
However, $L_K(E)$ does not have simultaneously the directly finite, stably finite, Hermite and cancellation of projectives
properties, by Theorem~\ref{thm5.2}.   \hfill $\Box$ 
\end{example}

\begin{rem}\label{summary}
In summary, we recall the hierarchy of   five cancellation properties of rings presented in the Introduction.     We have established that, within the class of Leavitt path algebras,  the IBN property is strictly weaker than the UGN property; the UGN property is strictly weaker than the stably finite property; and the stably finite, Hermite, and cancellation of projective properties are equivalent.   Moreover, the graphs $E$ for which $L_K(E)$ has the UGN property, and the graphs $F$ for which $L_K(F)$ has any one of the final three properties, have been explicitly described.

It remains an open question to give graph-theoretic conditions on $E$ which describe precisely the Leavitt path algebras $L_K(E)$ having the IBN property.    \hfill $\Box$ 
\end{rem}

%As a corollary of Theorem~\ref{thm5.2}, 
We finish this paper by giving a description of the Cohn
path algebras of finite graphs that have any one of the above properties.

\begin{cor}\label{cor5.4}
Let $E$ be a finite graph and $K$ any field. Then the following  are equivalent:
\begin{enumerate}
\item[(1)] $C_K(E)$ has cancellation of projectives;
\item[(2)] $C_K(E)$ is Hermite;
\item[(3)] $C_K(E)$ is stably finite;
\item[(4)] $C_K(E)$ is directly finite;
\item[(5)] $E$ is acyclic.
\end{enumerate}
\end{cor}

\begin{proof}   We  have that $C_K(E)\cong L_K(F(E))$, where $F(E)$ is the graph constructed from $E$ given near the end of Section \ref{LpashaveUGNSection}.  Using that description,   it is easy to see that
$F(E)$ is a no-exit graph if and only if $E$ is acyclic. The result then follows immediately from 
Theorem~\ref{thm5.2}.
\end{proof}

 \vskip 0.5cm {

\end{document}
\begin{thebibliography}{99}

\bibitem{ap:tlpaoag05} G.~Abrams and G.~Aranda Pino, The Leavitt path
algebra of a graph, \emph{Journal of Algebra}, \textbf{293} (2005), 319--334.

%\bibitem{ap:tlpaoag08} G.~Abrams and G.~Aranda Pino, The Leavitt path
%algebras of arbitrary graphs, \emph{Houston Journal of Mathematics},
%\textbf{34} (2008), 423--442.

\bibitem{aam:lpa} G. Abrams, P. Ara, and M. Siles Molina, \emph{Leavitt path
algebras}.  Lecture Notes in Mathematics series, Springer-Verlag Inc. (to appear).

\bibitem{ak:cpahibn}
G. Abrams and M. Kanuni, Cohn path algebras have invariant
basis number, \emph{Comm. Algebra}, \textbf{44} (2016), 371 - 380.

\bibitem{apm:lflpa}
G. Abrams, G. Aranda Pino and M. Siles Molina, Locally finite
Leavitt path algebras, \emph{Israel J. Math}, \textbf{165} (2008),
329-348.

%\bibitem{aps:fiitcofpa}
%G. Abrams, A. Louly, E. Pardo, C. Smith, Flow invariants in the classification of Leavitt path algebras,
%\emph{J. Algebra}, \textbf{333} (2011), 202--231.

\bibitem{ag:lpaosg} P.~Ara and K. Goodearl, Leavitt path algebras of separated
graphs, \emph{J. Reine Angew. Math.}, \textbf{669} (2012), 165--224.

\bibitem{amp:nktfga} P.~Ara, M.\,A.~Moreno, E.~Pardo, Nonstable
K-theory for graph algebras, \emph{Algebr.\ Represent.\ Theory},
\textbf{10} (2007), 157--178.

\bibitem{ar:fpsmolpa} P. Ara and K.M. Rangaswamy, Finitely presented simple
modules over Leavitt path algebras, \emph{J. Algebra},
\textbf{417} (2014), 333--352.

%\bibitem{al:nlpaatra} G. Aranda Pino and L. Va\v{s}, Noetherian Leavitt
%path algebras and their regular algebras, \emph{Mediterr. J. Math.},
%\textbf{10} (2013), 1633--1656.

%\bibitem{bass:spicakiakt}
%H. Bass, \emph{Some prolems in ``classical" algebraic K-theory, in
%``Algebraic K-Theory II"}, Lecture Notes in Math., Vol. 342, pp. 3--73,
%Springer-Verlag, New York-Berlin, 1972.

%\bibitem{bh:ctopmoatraipe}
%S. M. Bhatwadekar, Cancellation theorems of projective modules over
%a two-dimension ring and its polynomial extensions,
%\emph{Compositio Mathematica}, \textbf{128} (2001), 339--359.

\bibitem{cs:othtoctc} W. H. Cockcroft and R. G. Swan, On the homotopy type of
certain two-dimensional complexes, \emph{Proc. Lond. Math.  Soc.}, \textbf{11}
(1961), 193--202.

\bibitem{c:srotibp} P. M. Cohn, Some remarks on the invariant basis property, \emph{Topology}, \textbf{5} (1966), 215--228.

%\bibitem{c:rfor}
%P. M. Cohn, Rank functions on rings, \emph{Journal of Algebra},
%\textbf{133} (1990), 373--385.

\bibitem{c:fhrtsd}
P. M. Cohn, From Hermite rings to Sylvester domains, \emph{Proc.
Amer. Math. Soc}, \textbf{128} (2000), 1899--1904.

\bibitem{c:acfartbp}
P. M. Cohn, Another criterion for a ring to be projective-free,
\emph{Bull. London Math. Soc}, \textbf{37} (2005), 857--859.

\bibitem{c:firaligr}
P. M. Cohn, \emph{Free ideal rings and localization in general
rings}.  Cambridge University Press, New York, 2006.

%\bibitem{g:vnrr} K. R. Goodearl, \emph{von-Neumann regular rings},
%Pitman, London, 1979.

%\bibitem{g:srrarf} K. R. Goodearl, Simple regular rings and rank functions,
%\emph{Math. Ann.}, \textbf{214} (1975), 267 - 278.

%\bibitem{gh:rfakorr} K. R. Goodearl and D. Handelman, Rank functions and
%$K_0$ of regular rings, \emph{J. Pure Appl. Algebra}, \textbf{7} (1976), 195 - 216.

\bibitem{hv:ibnarpfr} A. Haghany and K. Varadarajan, IBN and related
properties for rings, \emph{Acta Math. Hungar.}, \textbf{94} (2002), 251 - 261.

%\bibitem{l:sc}
%T. Y. Lam, \emph{Serre's conjecture}, Lecture Notes in Math., Vol. 635,
%Springer-Verlag, New York-Berlin, 1978.

\bibitem{l:lomar} T. Y. Lam, \emph{Lectures on modules and rings}. Springer-Verlag, New York-Berlin, 1999.

\bibitem{l:spopm}
T. Y. Lam, \emph{Serre's problem on projective modules}. Springer
Monographs in Mathematics, Springer-Verlag, Berlin, 2006.

%\bibitem{leav:mwibn}
%W.G. Leavitt, Modules without invariant basis number, \emph{Proc. Amer. Math.
%Soc.},  \textbf{8} (1957), 322 --328.

\bibitem{leav:tmtoar}
W.\,G.~Leavitt, The module type of a ring, \emph{Trans.
Amer.\ Math.\ Soc.}, \textbf{42} (1962), 113--130.

%\bibitem{leav:tmtohi}
%W.\,G.~Leavitt, The module type of homomorphic images,
%\emph{Duke Math.\ J.}, \textbf{32} (1965), 305--311.

%\bibitem{m:itakt}
%J. W. Milnor, \emph{Introduction to Algebraic K-theory}, Ann. Math. Studies
%No. 72. Princeton University Press, Princeton, NJ 1971.

\bibitem{r:ga}
I. Raeburn, \emph{Graph Algebras}. in: CBMS Regional Conference Series in Mathematics, Vol. 103,
American Mathematical Society, Providence, RI, 2005, vi+113 pp. Published for the Conference
Board of the Mathematical Sciences, Washington, DC.

%\bibitem{swan:pmolpr}
%R.\,G.~Swan, Projective modules over Laurent polynomial rings, \emph{Trans.
%Amer.\ Math.\ Soc.}, \textbf{237} (1978), 111--120.

\bibitem{tomf:lpawciacr} M. Tomforde, Leavitt path algebras with
coefficients in a commutative ring, \emph{J.~Pure Appl.\ Algebra},
\textbf{215} (2011), 471--484.

\bibitem{Vas}  L. Va\v{s},  Canonical traces and directly finite Leavitt path algebras, \emph{Alg. Rep. Theory} \textbf{18} (2015), 711 -- 738.  

%\bibitem{y:thrcdo} I. Yengui, The Hermite ring conjecture in
%dimension one, \emph{Journal of Algebra},
%\textbf{230} (2008), 447--441.
\end{thebibliography}
